\pdfoutput=1
\RequirePackage{ifpdf}
\ifpdf 
\documentclass[pdftex]{sigma}
\else
\documentclass{sigma}
\fi

\usepackage[all,2cell]{xy} \UseAllTwocells \SilentMatrices

\newtheorem{thm}{Theorem}[section]
\newtheorem{cor}[thm]{Corollary}

\newtheorem{lem}[thm]{Lemma}
\newtheorem{prop}[thm]{Proposition}

\theoremstyle{definition}
\newtheorem{defn}[thm]{Definition}
\newtheorem{Example}[thm]{Example}

\newtheorem{rem}[thm]{Remark}

\numberwithin{equation}{section}

\usepackage{bbm}

\def\C{{\mathbbm C}}
\def\N{{\mathbbm N}}
\def\R{{\mathbbm R}}
\def\Z{{\mathbbm Z}}

\def\1{{\mathbbm{1}}}

\newcommand{\Hom}{\operatorname{Hom}}
\newcommand{\HOM}{\operatorname{HOM}}
\renewcommand{\to}{\rightarrow}

\def\dif{\partial}

\def\lra{{\longrightarrow}}
\def\dmod{{\mathrm{\mbox{-}mod}}} 
\def\pmod{{\mathrm{\mbox{-}pmod}}} 
\def\Id{\mathrm{Id}}
\def\mc{\mathcal}

\def\shuffle{\,\raise 1pt\hbox{$\scriptscriptstyle\cup{\mskip
 -4mu}\cup$}\,}
\newcommand{\refequal}[1]{\xy {\ar@{=}^{#1}
(-1,0)*{};(1,0)*{}};
\endxy}

\newcommand{\oplusop}[1]{{\mathop{\oplus}\limits_{#1}}}

\newcommand{\mH}{\mathrm{H}} 

\newcommand{\Tot}{\mathrm{Tot}}

\begin{document}

\allowdisplaybreaks

\newcommand{\arXivNumber}{1911.02503}

\renewcommand{\thefootnote}{}

\renewcommand{\PaperNumber}{019}

\FirstPageHeading

\ShortArticleName{A Faithful Braid Group Action on the Stable Category of Tricomplexes}

\ArticleName{A Faithful Braid Group Action\\ on the Stable Category of Tricomplexes\footnote{This paper is a~contribution to the Special Issue on Algebra, Topology, and Dynamics in Interaction in honor of Dmitry Fuchs. The full collection is available at \href{https://www.emis.de/journals/SIGMA/Fuchs.html}{https://www.emis.de/journals/SIGMA/Fuchs.html}}}

\Author{Mikhail KHOVANOV~$^\dag$ and You QI~$^\ddag$}

\AuthorNameForHeading{M.~Khovanov and Y.~Qi}

\Address{$^\dag$~Department of Mathematics, Columbia University, New York, NY 10027, USA}
\EmailD{\href{mailto:khovanov@math.columbia.edu}{khovanov@math.columbia.edu}}
\URLaddressD{\url{https://www.math.columbia.edu/~khovanov/}}

\Address{$^\ddag$~Department of Mathematics, University of Virginia, Charlottesville, VA 22904, USA}
\EmailD{\href{mailto:yq2dw@virginia.edu}{yq2dw@virginia.edu}}
\URLaddressD{\url{https://math.virginia.edu/people/yq2dw/}}

\ArticleDates{Received November 07, 2019, in final form March 19, 2020; Published online March 29, 2020}

\Abstract{Bicomplexes of vector spaces frequently appear throughout algebra and geo\-metry. In Section~2 we explain how to think about the arrows in the spectral sequence of a~bicomplex via its indecomposable summands. Polycomplexes seem to be much more rare. In Section~3 of this paper we rethink a well-known faithful categorical braid group action via an action on the stable category of tricomplexes.}

\Keywords{braid group; categorical action; bicomplexes; spectral sequence; tricomplexes; stable category}

\Classification{20F36; 18G05; 18G40}

\begin{flushright}
\it To Dmitry Fuchs on his 80th birthday
\end{flushright}

\renewcommand{\thefootnote}{\arabic{footnote}}
\setcounter{footnote}{0}

\section{Introduction}
\begin{flushright}
\parbox{3.5in}{\small ``The impact of spectral sequences on algebraic topology was tremendous: Many major problems of topology, both solved and unsolved, became exercises for
students \dots'' \\
A.~Fomenko and D.~Fuchs~\cite[Preface]{FF}}
\end{flushright}

Representation theory, which has been established for over a century, deals with linear actions of groups and algebras. Much more recent is the discovery of interesting categorical actions of groups, primarily discrete groups. In these examples discrete
groups act by symmetries of categories, which in many cases are triangulated, and
the action preserves the triangular structure. One of the first nontrivial
examples appeared in~\cite{KhSe}, see also~\cite{KH, ST}. There the $n$-strand braid group $\mathrm{Br}_n$ acts on the homotopy category of complexes of modules over a particular finite-dimensional algebra $A_{n-1}$. The action is by exact functors and on the Grothendieck group the action descends either to the Burau representation of the braid group (if one keep tracks of an additional grading on modules, in addition to the homological grading) or
to the reduced permutation action of the symmetric group. Neither of these linear actions is faithful, but its categorical lifting was shown to be faithful in~\cite{KhSe}.

The algebra $A_{n-1}$ (the \emph{zigzag} algebra) is the quotient of the path algebra of the quiver with $n-1$ vertices and edges connecting adjacent vertices in both directions (assuming $n> 3$, with minor changes necessary for $n=2,3$)
\begin{gather*}
\xymatrix{
 \overset{1}{\circ} &
 \cdots \ltwocell{'}&
 \overset{i-1}{~\circ~}\ltwocell{'}&
 \overset{i}{\circ}\ltwocell{'}&
 \overset{i+1}{~\circ~}\ltwocell{'}&
 \cdots \ltwocell{'}&
 \overset{n-1}{\circ}\ltwocell{'}
 }
\end{gather*}

Generators of $A_{n-1}$ corresponding to arrows in the quiver are
denoted $(i|i\pm 1)$. The defining relations
\[
(i|i+1|i+2)=0,\qquad (i|i-1|i-2)=0,\qquad (i|i-1|i)=(i|i+1|i)
\]
(for $i$'s for which both sides of a relation make sense) are quadratic,
$A_{n-1}$ is finite dimensional, with a basis consisting of idempotents $(i)$,
edges $(i,i\pm 1)$ and length two paths $(i|i\pm 1|i)$. For $1< i < n$ indecomposable projective $A_{n-1}$ module $P_i = A_{n-1}(i)$ is four-dimensional, with the basis $\{(i),(i-1|i),(i+1|i),(i|i-1|i)\}$ and
can be visualized as a diamond.
\[
\xymatrix{
& (i|i-1|i) & \\
(i-1|i) \ar[ur]^{(i|i-1)} && (i+1|i). \ar[ul]_{(i|i+1)}\\
& (i)\ar[ur]_{(i+1|i)} \ar[ul]^{(i-1|i)} &\\
}
\]

The defining relations in $A_{n-1}$ can be interpreted as the defining relations in the category of bicomplexes. Namely, let
\[
\partial_1 = \sum_{i=1}^{n-2}(i|i+1), \qquad
\partial_2 = \sum_{i=2}^{n-1}(i|i-1).
\]
Then the defining relations in $A_{n-1}$ can be rewritten as
\[ \partial_1^2=0, \qquad \partial_2^2=0, \qquad \partial_1\partial_2 =
\partial_2\partial_1.
\]
These are exactly the relations on the two differentials in a bicomplex. A bicomplex is built out of vector spaces placed in the vertices of an integral lattice~$\Z^2$, with the differentials going along the two coordinates, with the unit step each.

One can introduce a grading on $A_{n-1}$ by making, for instance, left-pointing arrows (edges) in the quiver to have degree one and right-pointing edges degree zero. The unit element of~$A_{n-1}$ decomposes as the sum of $n-1$ idempotents, one for each vertex of the graph,
$1 = (1)+(2)+\dots + (n-1)$, inducing the decomposition of an $A_{n-1}$-module into a sum of vector spaces
\begin{gather*} M = \bigoplus_{i=1}^{n-1} (i) M ,\end{gather*}
and the additional grading on $M$ leads to the bigrading, with the left and right directed edges changing the bigrading by $(1,0)$ and $(0,1)$, respectively.

In this way, graded $A_{n-1}$-modules may be identified with bicomplexes with nonzero terms restricted to a suitable area of the lattice $\Z^2$. Changing the indexing of quiver vertices from $\{1,2,\dots, n-1\}$ to $\Z$ by passing
to the quiver that is infinite in both directions (see figure in equation~(\ref{quiver-Q}))
results in a non-unital algebra $A_{\infty}$ with a system of idempotents
$\{(i)\}_{i \in \Z}$ such that graded $A_{\infty}$-modules naturally correspond to bicomplexes.

The braid group $\mathrm{Br}_n$ acts on the homotopy category of (either graded or ungraded) $A_{n-1}$-modules by tensoring with a suitable complex of $A_{n-1}$-bimodules. This works as well in the limit of $A_{\infty}$-modules, with the braid group $\mathrm{Br}_\infty$ with strands (and generators $\sigma_i$) enumerated by integers.

Passing from modules over an algebra $B$ to complexes of modules means working with suitably graded modules over the algebra $B[d]/\big(d^2\big)$. In our case, graded $A_{n-1}$ or $A_{\infty}$ modules can be identified with bicomplexes (more precisely, there is an equivalence of corresponding abelian categories). Consequently, complexes of
$A_{n-1}$ and $A_{\infty}$-modules may be identified with tricomplexes, with the homological
grading in $A_{n-1}[d]/\big(d^2\big)$ corresponding to the additional,
third, grading in tricomplexes.

Passing from complexes to the homotopy category of complexes (of modules over an algebra~$B$) means modding out by null-homotopic morphisms. If one restricts to complexes of projective $B$-modules, which is a common and important subcategory of the category of complexes, this means killing morphisms which factor through a direct sum of objects of the form
\[
0 \lra B \stackrel{\mathrm{Id}}{\lra} B \lra 0
\]
in various homological degrees. Specializing~$B$ to~$A_{\infty}$, the above complex decomposes as a direct sum of terms of the form
\begin{gather}\label{exact-comp}
 0 \lra A_{\infty}(i) \stackrel{\mathrm{Id}}{\lra} A_{\infty}(i) \lra 0
\end{gather}
for various $i\in \Z$ (cf.\ the next diagram below). By keeping track of the additional grading, one can further shift these copies of $A_{\infty}(i)$ and parametrize them by a pair of integers $(i,j)$. Together with the homological grading~$k$, one gets a 3-parameter family of possible indecomposable summands that each represent the zero complex in the homotopy category.

If this setup is converted into the language of bicomplexes and tricomplexes, the module~$A_{\infty}(i)$ corresponds to a free rank one bicomplex in the bigrading associated to the idempotent~$(i)$ and another independent grading~$j$, see Definition~\ref{def-algebra-A}. The complex~(\ref{exact-comp}) corresponds to a free rank one tricomplex, with its generator placed in tridegree $(i,0,0)$. We refer the reader to~\eqref{eqn-G-commute-with-grading-shifts} and~\eqref{eqn-G-on-Pr} for the precise matching of trigradings and shifts:
\[
\xymatrix@!=2.3pc{
(i|i-1|i) \ar[rr]^{\Id} && (i|i-1|i) &\\
& (i+1|i)\ar[rr]^(0.4){\Id} \ar[ul]^{\dif_1} && (i+1|i) \ar[ul]^{\dif_1} \\
(i-1|i)\ar[uu]^{\dif_2} \ar'[r][rr]^(-0.25){\Id} && (i-1|i) \ar'[u][uu]^{\dif_2} &\\
 & (i) \ar[rr]^{\Id} \ar[uu]^(0.3){\dif_2} \ar[ul]^{\dif_1} && (i) \ar[uu]^{\dif_2} \ar[ul]^{\dif_1}\\
}
\]
Here in the diagram, the (basis elements of the) first copy of $A_\infty(i)$ in \eqref{exact-comp} is exhibited as the left-most square in the cube, while the second copy of $A_\infty$ is displayed as the right-most square. These two squares are connected by the homological differential labelled with $\mathrm{Id}$ maps.

In the homotopy category of projective graded $A_{\infty}$-modules, a morphism is zero if it factors through the object which is a direct sum of complexes (\ref{exact-comp}) over various $i$, $j$, and $k$, where $i$ labels the idempotent, $j$ is the additional grading parameter in $A_{\infty}$, and $k$ is the homological grading. Converting this to tricomplexes, one unites the three integer grading parameters~$i$,~$j$,~$k$ of different origins into a single trigrading on tricomplexes. The complex~(\ref{exact-comp}) becomes a free tricomplex of rank one that can sit in any position relative to the trigrading. Killing morphisms that factor through sums of such free rank one tricomplexes is equivalent to the condition that one is working in the stable category of tricomplexes, that is in the category of tricomplexes modulo the ideal of morphisms that factor through a free tricomplex.

Tricomplexes can be described as trigraded modules over the algebra
$\Lambda_3$ with generators $\partial_1$,~$\partial_2$,~$\partial_3$ and relations
\begin{gather*} \partial_i^2 =0, \qquad i=1,2,3, \qquad \partial_i\partial_j = \partial_j\partial_i , \qquad i\not= j.
\end{gather*}
This 8-dimensional algebra is Frobenius, and it is even a Hopf algebra in the category of super-vector spaces. Consequently, its stable category of trigraded modules is triangulated (and monoidal, due to the Hopf algebra structure).

The braid group action on the homotopy category of $A_{n-1}$ and $A_{\infty}$-modules transfers to the stable category of tricomplexes. Note that the homotopy category of $A_{\infty}$-modules and the stable category of tricomplexes are not equivalent, but rather admit equivalent subcategories with matching actions of the braid group. On the $A_{\infty}$ side, it is the homotopy category of complexes of projective modules, and on the tricomplex side, the stable subcategory generated by tricomplexes that restrict to free bicomplexes relative to the subalgebra generated by differentials $\partial_1$, $\partial_2$.
The braid group action respects these subcategories and the equivalence between them.

The braid group acts by exact functors on this triangulated category of tricomplexes. The actions does not respect the monoidal structure, though, and choosing the action requires singling out one differential out of three. Choosing different differentials gives three commuting braid group actions.

For now, we view this example as a curiosity. One natural question is whether our example fits into the more general framework of Hopfological algebra~\cite{Kh3,Q}, where stable categories of modules over Hopf algebras, such as~$\Lambda_3$, are used as base categories for new constructions of categorifications (see, e.g.,~\cite{KQ}) or, perhaps, some other algebro-geometric structures. Another open problem is whether homotopy categories of complexes over other algebras of importance in categorification, such as arc algebras~\cite{Kh2}, can be rethought through some generalization of the stable category of tricomplexes.

Tricomplexes seem to appear exceedingly rarely in mathematics. Currently, they have made appearances in the BRST theory~\cite{Sh}, in the deformation theory of Hopf algebras~\cite{Y}, and in the algebraic K-theory~\cite{Ca}. A modified notion of a tricomplex, called quasi-tricomplex, occurs in the theory of variation~\cite{O}.

Beyond tricomplexes, polycomplexes can be related to $(\C^{\ast})^n$-equivariant coherent sheaves on~$\mathbb{CP}^{n-1}$ via a version of Beilinson--Gelfand--Gelfand--Koszul duality.

The braid group action on the stable category of tricomplexes is constructed in Section~\ref{sectricomplex} of this paper. In Section~\ref{secbicomplex} we explain a way to think about arrows in the spectral sequence of a bicomplex of vector spaces via indecomposable modules over the rings $A_{n-1}$ and $A_{\infty}$. This relation was independently discovered by Stelzig~\cite{St-2018}.

\section{Spectral sequences via indecomposable bicomplexes} \label{secbicomplex}

\begin{flushright}
\parbox{3.5in}{\small ``The subject of spectral sequences is elementary,
 but the notion of the spectral sequence of a double complex involves so many
objects and indices that it seems at first repulsive.'' D.~Eisenbud~\cite[Appendix 3.13]{Eisenbud}}
\end{flushright}

The standard textbook approach to spectral sequences makes them seem sophisticated and mysterious gadgets~\cite{CE, GM, McCleary,Weibel, W} and \cite[Appendix 3.13]{Eisenbud}.
Timothy Chow, in the introduction to his article on spectral sequences~\cite{Chow},
quotes the opinions of experts who, essentially, say that the definition is so complicated that you just have to get used to it.

The goal of this section is to explain spectral sequences, restricted to
bicomplexes of vector spaces, in a simple and straightforward way. Most of this section has appeared in lectures to graduate students by the first author, see for instance the informal lecture notes~\cite{Kh-2010}. Similar results also appeared in Stelzig~\cite{St-2018}. We warn the reader that this elementary approach works only for a bicomplex of vector spaces. Bicomplexes and filtered complexes that appear in spectral sequences in algebraic topology carry an enormous amount of extra structure, such as an action of the Steenrod algebra when working over~$\Z/p$, and cannot be easily understood in this elementary way. The complexity and beauty of these structures are captured in the Fomenko and Fuchs classic~\cite{FF} and other books, see McCleary~\cite{McCleary}.

\subsection{Cohomology}
Let $\Bbbk$ be a field and $M$
\begin{gather*} \dots \stackrel{d^{i-1}}{\lra} M^i \stackrel{d^i}{\lra} M^{i+1} \stackrel{d^{i+1}}{\lra} \cdots \end{gather*}
a complex of $\Bbbk$-vector spaces. We allow unbounded complexes and infinite-dimensional
vector spaces. It is easy to see that $M$ decomposes into
the direct sum of length zero complexes
\begin{gather*} 
0 \lra \mH^i \lra 0,
\end{gather*}
with a vector space $\mH^i$ in degree~$i$, and length one complexes
\begin{gather}\label{eqncontractiblesummand}
 0 \lra W^i \stackrel{\mathrm{id}}{\lra} W^i \lra 0,
\end{gather}
with two copies of a vector space $W^i$ in degrees $i$ and $i+1$. Thus,
 \begin{gather*} M^i \cong \mH^i \oplus W^i \oplus W^{i-1},\end{gather*}
although this direct sum decomposition of the vector space $M^i$ is not canonical. The
inclusion of the direct sum $\mH^i \oplus W^{i-1} \subset M^i$ is canonical, being
the inclusion $\operatorname{ker}(d^i) \subset M^i$.
The $i$-th cohomology group $\mH^i(M)$ of $M$ is canonically isomorphic to~$\mH^i$. Direct summands \eqref{eqncontractiblesummand} are contractible (recall that a complex is called \emph{contractible} if the identity endomorphism is null-homotopic).

\begin{Example}
Let $X$ be a smooth compact manifold and $(\Omega(X), d)$ the
de Rham complex of smooth forms on $X$. In this case $\mH^i(X,\R)$
are finite-dimensional vector spaces, while the vector spaces~$\Omega^i(X)$
and hence $W^i$ are infinite-dimensional. The bulk of the complex~$\Omega(X)$
is occupied by contractible ``junk'', while the ``valuable part'' (cohomology) has
small size.
If we equip $X$ with a Riemannian metric $g$, the operator $d^{\ast}=\pm \ast d\ast$
adjoint to $d$ gives rise to the Laplace operator
\begin{gather*}\Delta\colon \ \Omega^i(X) \lra \Omega^i(X), \qquad
\Delta= dd^{\ast} + d^{\ast} d.\end{gather*}
The Laplace operator provides a \emph{canonical} embedding of each complex
$ 0 \lra \mH^i(X,\R) \lra 0 $ into the complex $(\Omega(X), d)$, via the isomorphism
 $\mH(X,\R) \cong \operatorname{ker}(\Delta)$.
\end{Example}

A complex of $\Bbbk$-vector spaces is the same as a graded module over the exterior
$\Bbbk$-algebra $\Lambda_1$ on one generator~$d$ of degree~$1$:
\begin{gather*} \Lambda_1 := \Bbbk [d]/\big(d^2\big). \end{gather*}
The $i$-th homogeneous piece of a graded $\Lambda_1$-module $M$ is
a vector space~$M^i$, and the action of~$d$ is exactly the differential
$d\colon M^i \lra M^{i+1}$.

The category $\mc{M}_1$ of graded modules over $\Lambda_1$ is Krull--Schmidt, and
any module (even infinite-dimensional) decomposes into a direct sum of
indecomposable modules $S^i$ and $P^i$. Here $S^i$ is the one-dimensional
$\Bbbk$-vector space placed in degree $i$, and corresponds to the complex
$0 \lra \Bbbk \lra 0$. The differential acts by $0$ and the module $S^i$ is simple.
The module $P^i = \Lambda_1\{ i\}$ is free and corresponds to the complex
\begin{gather*} 0 \lra \Bbbk \stackrel{1}{\lra} \Bbbk \lra 0. \end{gather*}
Thus,
\begin{gather*} M \cong \oplusop{i\in \Z}\big(\mH^i \otimes S^i\big) \oplus \big(W^i\otimes P^i\big),\end{gather*}
and the cohomology of~$M$ only catches the first terms in the sum. Recall that an object~$M$ of an additive category is called \emph{indecomposable} if~$M$ is not isomorphic to a direct sum $N_1\oplus N_2$ with both $N_1$, $N_2$ nontrivial.

\subsection{Bicomplexes}\label{subsec-bicomplexes}
Let us now move on to bicomplexes. A bicomplex $M$ over a field $\Bbbk$ is
a family $\big\{M^{i,j}\big\}$ of vector spaces, for $i,j\in \Z$, and maps
\begin{gather*}\dif_1\colon \ M^{i,j} \lra M^{i+1,j}, \qquad \dif_2\colon \ M^{i,j}\lra M^{i,j+1}\end{gather*}
subject to the equations
\begin{gather*} \dif_1^2 =0, \qquad \dif_2^2=0, \qquad \dif_1 \dif_2 + \dif_2 \dif_1 =0,\\
\xymatrix{
& \vdots & \vdots & & \\
\ar[r]^-{\dif_1} & M^{i,j+1} \ar[r]^{\dif_1} \ar[u]^-{\dif_2} & M^{i+1,j+1} \ar[r]^-{\dif_1} \ar[u]^{\dif_2} & \\
\ar[r]^-{\dif_1} & M^{i,j} \ar[r]^-{\dif_1} \ar[u]^-{\dif_2} & M^{i,j+1} \ar[r]^-{\dif_1} \ar[u]^{\dif_2} & \\
& \ar[u]^-{\dif_2} & \ar[u]^-{\dif_2} & &
}
\end{gather*}

Let $\Lambda_2$ be the exterior $\Bbbk$-algebra on two generators $\dif_1$, $\dif_2$, so that the above equations are the defining relations for the generators. $\Lambda_2$ has a natural bigrading by
\begin{gather*} \deg(\dif_1) = (1,0), \qquad \deg(\dif_2) = (0,1).\end{gather*}
A bicomplex~$M$ is the same as a bigraded left $\Lambda_2$-module. We denote the category of bicomplexes by~$\mc{M}_2$.

We say that a bicomplex $M$ is \emph{bounded} if only finitely many $M^{i,j}$ are not~$0$.

\begin{Example}\label{egindecomposableLambda2mod}
Let us describe some bounded indecomposable bicomplexes.
\begin{enumerate}\itemsep=0pt
\item[(1)] The bicomplex $S^{i,j}$ is one-dimensional with a copy of $\Bbbk$ sitting in the $(i,j)$-th bidegree:
\[
S^{i,j}\colon \
\begin{gathered}
\xymatrix{
& 0 & \\
0\ar[r] & \Bbbk \ar[r]\ar[u] & 0.\\
& 0\ar[u] &
}
\end{gathered}
\]
In other words, $S^{i,j}$ is the simple $\Lambda_2$-module sitting in bidegree $(i,j)$.
\item[(2)] The indecomposable bicomplex $P^{i,j} \cong \Lambda_2\{i,j\}$ is a free rank one module (looking like a~square on a planar lattice), a copy of~$\Lambda_2$ with bigrading shifted, so that the nonzero term in the southwest corner sits in $(i,j)$-th degree:
\[
P^{i,j}\colon \
\begin{gathered}
\xymatrix{
 & 0 & 0 & \\
0 \ar[r] & \Bbbk\ar[r]\ar[u] & \Bbbk \ar[r]\ar[u] & 0 \\
0 \ar[r] & \Bbbk\ar[r]\ar[u] &\Bbbk \ar[r]\ar[u] & 0.\\
 & 0\ar[u] & 0 \ar[u] &
}
\end{gathered}
\]
\item[(3)]
The bicomplex $Z^{i,j}_{\rightarrow,l}$ has the top leftmost term
in bidegree $(i,j)$ and goes zigzag to the right and down. The number $l\in \N$ denotes
the number of nonzero arrows, $l+1$ is the dimension of the
vector space underlying this bicomplex
\[
Z^{i,j}_{\rightarrow,l}\colon \
\begin{gathered}
\xymatrix{
 \Bbbk \ar[r] & \Bbbk & & \\
 & \Bbbk\ar[u] \ar[r] & \Bbbk & \\
 & & \ddots \ar[r]\ar[u] & \Bbbk\\
 & & & \Bbbk. \ar[u]
}
\end{gathered}
\]
\item[(4)] Likewise, the bicomplex $Z^{i,j}_{\uparrow,l}$ starts from the bidegree $(i,j)$ and goes zigzag down and to the right.
 \[
Z^{i,j}_{\uparrow,l}\colon \
\begin{gathered}
\xymatrix{
 \Bbbk & & & \\
 \Bbbk\ar[u] \ar[r] & \Bbbk & & \\
 & \ddots \ar[r]\ar[u] & \Bbbk &\\
 & & \Bbbk \ar[u] \ar[r] & \Bbbk.
}
\end{gathered}
\]

\end{enumerate}
\end{Example}

\begin{thm} \label{thmclassbicomplex}
Any bounded bicomplex $M\in \mc{M}_2$ $($possibly with infinite-dimensional
vector spaces $M^{i,j})$ breaks up into a direct sum of indecomposable bicomplexes
$S^{i,j}$, $P^{i,j}$, $Z^{i,j}_{\leftarrow,l}$, $Z^{i,j}_{\uparrow,l}$ described above.
\end{thm}

We will postpone the proof of the theorem until Section~\ref{subsecproof}.

Let $\Tot(M)$ be the total complex of the bicomplex $M$, with the
differential $d=\dif_1+\dif_2$ and the terms given by direct sums over
$M^{i,j}$ for $i+j$ fixed,
\begin{gather*}
 \cdots \stackrel{d}{\lra}
\Tot^k(M) \stackrel{d}{\lra} \Tot^{k+1}(M) \stackrel{d}{\lra} \cdots, \end{gather*}
where
$\Tot^k(M) = \oplusop{i+j=k} M^{i,j}$.

A common situation is that we want to compute the homology
of $\Tot(M)$ with respect to the differential $d$ and already
know the homology of $M$ with respect to, say, differential $\dif_2$
(the upward differential in our conventions).
These homology groups $\mH(M,\dif_2)$ are bigraded,
\begin{gather*} \mH(M,\dif_2) = \oplusop{i,j\in \Z} \mH^{i,j}(M,\dif_2) ,\end{gather*}
and we would like to understand the relation between them and
$\mH(\Tot(M),d)$.
If we write $M$ as a (possibly infinite) direct sum of indecomposable
bicomplexes $M_{\alpha}$, for $\alpha$ in some index set $A$,
then both $\mH(M,\dif_2)$ and $\mH(\Tot(M),d)$ decompose as direct sums
of cohomology groups of $M_{\alpha}$:
\begin{gather*}
 \mH(M,\dif_2) \cong \oplusop{\alpha\in A} \mH(M_{\alpha}, \dif_2), \\
 \mH(\Tot(M),d) \cong \oplusop{\alpha\in A} \mH(\Tot(M_{\alpha}), d).
\end{gather*}
Hence, we will compare $\mH(M,\dif_2)$ and $\mH(\Tot(M),d)$ for all types of indecomposable summands of $M$, case by case.

{\bf Case 1.} $S^{i,j}$ contributes a copy of $\Bbbk$ to $\mH^{i,j}(M,\dif_2)$
and a copy of $\Bbbk$ to $\mH^{i+j}(\Tot(M),d)$.

{\bf Case 2.} $\mH(P^{i,j},\dif_2)=0$ and $\mH(\Tot(P^{i,j}),d)=0$. Thus,
the ``square'' indecomposable bicomplex $P^{i,j}$ contributes nothing to both $\mH(M,\dif_2)$ and $\mH(\Tot(M),d)$.

For the module $Z^{i,j}_{\uparrow,l}$, there are two sub-cases.

{\bf Case 3.a.}
Firstly, let $l$, the number of nonzero arrows in the zigzag, be odd in $Z^{i,j}_{\uparrow,l}$,
 \[
Z^{i,j}_{\uparrow,l}\colon \
\begin{gathered}
\xymatrix{
 \Bbbk & & \\
 \Bbbk\ar[u] \ar[r] & \Bbbk & & \\
 & \ddots \ar[r]\ar[u] & \Bbbk \\
 & & \Bbbk. \ar[u]
}
\end{gathered}
\]
Cohomology of $Z^{i,j}_{\uparrow,l}$ with respect to the vertical
differential $\dif_2$ is zero. The total complex of this zigzag
has the form
\begin{gather*}0 \lra \Bbbk^r \stackrel{d}{\lra} \Bbbk^r \lra 0, \end{gather*}
where $d$ is an isomorphism and $2r = l+1$. Hence, cohomology
of the total complex is zero as well.

{\bf Case 3.b.} Suppose now that $l$ in $Z^{i,j}_{\uparrow,l}$ is even, $l=2r$,
 \[
Z^{i,j}_{\uparrow,l}\colon \
\begin{gathered}
\xymatrix{
 \Bbbk & && \\
 \Bbbk\ar[u] \ar[r] & \Bbbk & && \\
 & \ddots \ar[r]\ar[u] & \Bbbk & \\
 & & \Bbbk \ar[u] \ar[r] & \Bbbk.
}
\end{gathered}
\]
Cohomology with respect to $\dif_2$ produce a single $\Bbbk$ in
bidegree $(i+r,j-r)$. The total complex has the form
\begin{gather*} 0 \lra \Bbbk^r \stackrel{d}{\lra} \Bbbk^{r+1} \lra 0 \end{gather*}
with $d$ injective. Cohomology of the total complex is $\Bbbk$
in degree $i+j$ and zero elsewhere.

{\bf Case 4.a.}
For the module $Z^{i,j}_{\rightarrow,l}$, there are two sub-cases as well. We start with even $l=2r$,
\[
Z^{i,j}_{\rightarrow,l} \quad (l=2r)\colon \
\begin{gathered}
\xymatrix{
 \Bbbk \ar[r] & \Bbbk & & \\
 & \Bbbk\ar[u] \ar[r] & \Bbbk & \\
 & & \ddots \ar[r]\ar[u] & \Bbbk\\
 & & & \Bbbk. \ar[u]
}
\end{gathered}
\]
Cohomology with respect to $\dif_2$ give a single $\Bbbk$ in
bidegree $(i,j)$. The total complex is
\begin{gather*} 0 \lra \Bbbk^{r+1} \stackrel{d}{\lra} \Bbbk^r \lra 0\end{gather*}
with a surjective $d$, and it has cohomology $\Bbbk$
in degree $i+j$ and zero elsewhere.

Before we treat the last case, observe that in each of the above
cases cohomology of the total complex is given by simply collapsing
the bigrading of $\mH(M,\dif_2)$ into a single grading by adding~$i$ and~$j$. Thus, if $M$ does not contain any direct summands
isomorphic to $Z^{i,j}_{\rightarrow,l}$ with odd $l$,
\begin{gather*} \mH^k(\Tot(M), d) = \bigoplus_{i+j=k} \mH^{i,j}(M,\dif_2).\end{gather*}

{\bf Case 4.b.} Lastly, consider $Z^{i,j}_{\rightarrow,l}$ with odd $l=2r+1$,
\[
Z^{i,j}_{\rightarrow,l} \quad (l=2r+1)\colon \
\begin{gathered}
\xymatrix{
 \Bbbk \ar[r] & \Bbbk & && \\
 & \Bbbk\ar[u] \ar[r] & \Bbbk && \\
 & & \ddots \ar[r]\ar[u] & \Bbbk &\\
 & & & \Bbbk \ar[u] \ar[r]& \Bbbk.
}
\end{gathered}
\]
Taking cohomology with respect to $\dif_2$ produces two
copies of $\Bbbk$, in bigradings that differ by $(r,1-r)$:
\[
\mH(Z^{i,j}_{\rightarrow,l},\dif_2)=
\begin{gathered}
\xymatrix{
 \Bbbk \ar[r] & 0 & && \\
 & 0\ar[u] \ar[r] & 0 && \\
 & & \ddots \ar[r]\ar[u] & 0 &\\
 & & & 0 \ar[u] \ar[r]& \Bbbk.
}
\end{gathered}
\]
Collapsing the bigrading in cohomology gives us two copies of $\Bbbk$ in adjacent degrees $i+j$ and $i+j+1$.

The total complex has the form
\begin{gather*} 0 \lra \Bbbk^{r+1}\stackrel{d}{\lra} \Bbbk^{r+1}\lra 0 \end{gather*}
with $d$ an isomorphism, and the cohomology of the total complex
is zero.
Thus, for a general bounded bicomplex~$M$, the cohomology
$\mH(M,\dif_2)$, after the bigrading collapsed into a single
grading, is isomorphic to the cohomology of the total
complex of~$M$, plus pairs of copies of the ground field in adjacent degrees
$(i+j,i+j+1)$, for each direct summand of $M$ isomorphic to~$Z^{i,j}_{\rightarrow,l}$ with odd~$l$.

Since we want to know the cohomology of the total complex,
the extraneous terms need to be eliminated. Ideally, we would
locate all direct summands $Z^{i,j}_{\rightarrow,2r+1}$
and kill off pairs of $\Bbbk$, one for each summand, in the relative
bigrading position $(r,1-r)$. For a general $r$, we need to
eliminate pairs in the relative positions $(i,j)$ and $(i+r,j-r+1)$ by a map $d_r^{i,j}$:
\[
\begin{gathered}
\xymatrix{
 \Bbbk \ar[rrrrddd]^{d_r^{i,j}=1} & 0 & && \\
 & 0 & 0 && \\
 & & \ddots & 0 &\\
 & & & 0 & \Bbbk
}
\end{gathered}
\]
on the square lattice. This is exactly what the spectral sequence does.
The $E_1$-term of the spectral sequence of the bicomplex $(M,\dif_1,\dif_2)$
is the cohomology of $M$ with respect to $\dif_2$:
\begin{gather*}E_1^{i,j} = \mH^{i,j}(M,\dif_2).\end{gather*}
To pass to the $E_2$-term, we remove contributions to $\mH(M,\dif_2)$ from the
direct summands $Z^{i,j}_{\rightarrow, 1}$, which are
$\Bbbk\stackrel{1}{\rightarrow} \Bbbk$.
Notice that the $E_2$-term is simply the cohomology of $\mH(M,\dif_2)$ with
respect to the differential $\dif_1$ (more accurately, differential $\dif_1$ on $M$
descends to a differential on $\mH(M,\dif_2)$, which we also call $\dif_1$):
\begin{gather*} E_2 = \mH(\mH(M,\dif_2), \dif_1).\end{gather*}
Going from $E_2$ to the $E_3$-term, we remove pairs of one-dimensional vector
spaces $\Bbbk$
which come from summands $Z^{i,j}_{\rightarrow, 3}$ and differ by $(2,-1)$-bigrading.
In general, in the $E_r$-term there are no contributions from
summands $Z^{i,j}_{\rightarrow, l}$ for all odd $l\leq 2r-1$.

The reader can find an accurate definition of spaces $E_r^{i,j}$
and differentials $d_r^{i,j}$ in almost any textbook on homological algebra,
often done in a slightly different framework of a filtered complex rather than
a bicomplex. However, we find the above approach via indecomposable
bicomplexes more clarifying and intuitive than the standard textbook definition of the pages
$E_r$ and differentials $d_r^{i,j}$ of a spectral sequence.

\subsection{Bicomplexes and Hodge theory}
The Hodge bicomplex~\cite{DGMS, GH, W}.
Let $X$ be a closed almost complex manifold. This means $X$ is a smooth closed manifold equipped with an endomorphism $J$ of its real tangent bundle $T_\R(X)$ such that $J^2=-1$.
The complexified tangent bundle $T(X)=T_\R(X)\otimes_{\R}\C$ of
$X$ decomposes into the direct sum of $i$ and $-i$ eigenspaces of $J$,
\begin{gather*} T(X) = T^{1,0}(X) \oplus T^{0,1}(X).\end{gather*}
This induces a direct sum decomposition of all exterior powers $\wedge ^kT^{\ast}$
of the complexified cotangent bundle $T^{\ast}(X)_c$:
\begin{gather*} \wedge^kT^{\ast} = \bigoplus_{i+j=k} \wedge^{i,j}T^{\ast} .\end{gather*}
Let $\Omega_{\C}^k(X)$ be the space of smooth sections of $\wedge^kT^{\ast}$
and $(\Omega_{\C}(X),d)$ the complex with $d$ the complexified de Rham differential:
\begin{gather*}\cdots \stackrel{d}{\lra} \Omega_{\C}^k(X) \stackrel{d}{\lra} \Omega_{\C}^{k+1}(X)
\stackrel{d}{\lra}\cdots. \end{gather*}
Let $\Omega^{i,j}(X)$ be the vector space of smooth sections of
$\wedge^{i,j}T^{\ast}$.
In general, $d$ shows no respect for the direct sum decomposition
\begin{gather*} \Omega_{\C}^k(X) = \oplusop{i+j=k} \Omega^{i,j}(X).\end{gather*}
However, Newlander and Nirenberg proved~\cite{NN} that $d$ takes
$\Omega^{i,j}(X)$ to $\Omega^{i+1,j}(X)\oplus \Omega^{i,j+1}(X)$
for all $i,j$ if and only if the almost complex structure $J$ of $X$ comes from a
complex structure on $X$. In this case $d= \partial + \bar{\partial}$, where
\begin{gather*} \partial \colon \  \Omega^{i,j}(X) \lra \Omega^{i+1,j}(X)\end{gather*}
is the composition of $d$ with the projection onto the $(i+1,j)$-component,
and
\begin{gather*} \bar{\partial} \colon \ \Omega^{i,j}(X) \lra \Omega^{i,j+1}(X)\end{gather*}
is the composition of $d$ with the projection onto the $(i,j+1)$-component.
The relation $d^2=0$ splits into the relations
\begin{gather*} \partial^2=0, \qquad \bar{\partial}^2=0, \qquad
\partial \bar{\partial} + \bar{\partial} \partial =0.\end{gather*}
Thus, to a complex manifold $X$ there is assigned
the Hodge bicomplex $\big(\Omega_{\C}(X), \partial, \bar{\partial}\big)$.
Its cohomology groups with respect to $\bar{\partial}$ is known as
the \emph{Dolbeault cohomology}, while the cohomology with respect to
$d= \partial + \bar{\partial}$ is the \emph{de Rham cohomology of $X$}
with coefficients in $\C$. The spectral sequence of this bicomplex,
called the \emph{Hodge to de Rham spectral sequence},
has the Dolbeault cohomology as the $E_1$-term and converges to
the de Rham cohomology of~$X$.

Assume now that $X$ is a K\"ahler manifold. Then the $\partial\bar{\partial}$-lemma
holds.

\begin{lem}
If $\omega\in \Omega_{\C}(X)$ is a $d$-closed form and
either $\partial$-exact or $\bar{\partial}$-exact, then
\begin{gather*}\omega= \partial\bar{\partial} \alpha\end{gather*}
for some $\alpha\in \Omega_{\C}(X)$.
\end{lem}

Since the lemma is true for $\Omega_{\C}(X)$, it also holds for each
indecomposable summand of~$X$. A~simple examination shows that
the lemma fails for any zigzag $Z^{i,j}_{\rightarrow,l}$ and
$Z^{i,j}_{\uparrow,l}$ for $l>0$ (when $l=0$, the zigzag degenerates
to the simple bicomplex $S^{i,j}$). We obtain immediately the following.

\begin{prop} For a compact K\"ahler manifold $X$, every indecomposable summand of the bicomplex $\Omega_{\C}(X)$ is isomorphic to either~$S^{i,j}$ or $P^{i,j}$ for some~$i$,~$j$.
\end{prop}

Equivalently, $\Omega_{\C}(X)$ has no zigzags (including no zigzags of length $1$, that is $\Bbbk \stackrel{1}{\lra} \Bbbk$ and its vertical conterpart).

Thus, the bicomplex $\Omega_{\C}(X)$ decomposes into the direct sum
\begin{gather*}\Omega_{\C}(X)\cong \Omega_s(X) \oplus \Omega_p(X),\end{gather*}
where $\Omega_s(X)$ is a finite-dimensional semisimple bicomplex (a
direct sum of one-dimensional simple bicomplexes $S^{i,j}$), while
$\Omega_p(X)$ is an infinite-dimensional free bicomplex (a direct
sum of free bicomplexes $P^{i,j}$). The first summand is finite-dimensional since $\Omega_{\C}(X)$ has finite-dimensional cohomology groups, and \begin{gather*}\Omega_s(X)\cong \mH(\Omega_{\C}(X), \partial) \cong
\mH(\Omega_{\C}(X),\bar{\partial}) \cong \mH(\Omega_{\C}(X), d) \cong \mH(X,\C).\end{gather*}
The first three terms are bigraded vector spaces, and the second isomorphism
says that, after collapsing the bigrading to a single grading, the groups become
the usual de Rham cohomology groups of $X$.

We see that the cohomology groups of a compact K\"ahler manifold $X$ with respect to $\partial$, $\bar{\partial}$,
and $d$ are isomorphic; they are also isomorphic to the largest semisimple summand
of $\Omega_{\C}(X)$. The Hodge to de Rham spectral sequence for $X$
degenerates at $E_1$ ($E_1=E_{\infty}$). Likewise, the $\partial$ counterpart of the Hodge to de Rham spectral sequence degenerates at
 $E_1=\mH(\Omega_{\C}(X), \partial)$.

\subsection{Proof of Theorem \ref{thmclassbicomplex}}\label{subsecproof}
Let $M$ be a graded module over $\Lambda_2$. Suppose $m\in M$ is a homogeneous vector of bidegree $(i,j)$ such that $\dif_1\dif_2(m)\neq 0$. Then, it is clear that the submodule generated by $m$ and spanned by vectors in the diagram below
\[
\xymatrix{
\dif_2(m)\ar[r]^-{\dif_1} & \dif_2 \dif_1(m)=-\dif_1\dif_2(m)\\
m\ar[r]^-{\dif_1}\ar[u]^-{\dif_2} & \dif_1(m)\ar[u]^-{\dif_2}
}
\]
is isomorphic to $\Lambda_2$, up to a grading shift, and thus is a projective submodule inside $M$.

Recall that over a Frobenius algebra, any projective module is injective and vice versa \cite[Theorem~15.9]{La}. The same proof shows that over a graded Frobenius algebra, in the category of graded modules, projective objects coincide with injective objects.
Since $\Lambda_2$ is graded Frobenius, the submodule above is also graded injective, and therefore must be a direct summand of~$M$. We can decompose
$M\cong P\oplus N$, where $P$ is a graded direct sum of projective-injectives of the form $P^{i,j}$ (case (2) of Example \ref{egindecomposableLambda2mod}), and $N$ is annihilated by the element $\dif_1\dif_2=-\dif_2\dif_1\in \Lambda_2$. Further, we may regard $N$ as a module over the bigraded quotient algebra
\begin{gather*}\widehat{\Lambda}_2:=\dfrac{\Lambda_2}{(\dif_1\dif_2)}\cong \dfrac{\Bbbk[\dif_1,\dif_2]}{\big(\dif_1^2,\dif^2_2,\dif_1\dif_2\big)}.\end{gather*}

Now assume $N$ is a bounded bigraded $\Lambda_2$-module which does not contain any projective-injective summands. By the above discussion, $N$ is a bigraded module over $\widehat{\Lambda}_2^\prime$. Write for each term
\[
N^{i,j}\cong D^{i,j} \oplus C^{i,j},
\]
where
\begin{gather*}D^{i,j}=\operatorname{Ker}(\dif_1)\cap \operatorname{Ker}(\dif_2) \cap N^{i,j}\end{gather*}
is the subspace annihilated by both $\dif_1$ and $\dif_2$, and $C^{i,j}$ is an arbitrary complementary vector subspace to $D^{i,j}$ inside $N^{i,j}$. Necessarily,
\[
\dif_1\big(C^{i,j}\big)\subset D^{i+1,j},\qquad \dif_2\big(C^{i,j}\big)\subset D^{i,j+1}
\]
since $\dif_1\dif_2|_N\equiv 0$. Thus, there are two direct summands of $N$ containing the subspaces~$C^{i,j}$ and~$D^{i,j}$:
\begin{gather*}
\left(
\begin{gathered}
\xymatrix@=1.5em{
\ddots & && \\
 C^{i-1,j+1}\ar[u]^-{\dif_2} \ar[r]^-{\dif_1} & D^{i,j+1} && \\
 & C^{i,j} \ar[r]^-{\dif_1}\ar[u]^{\dif_2} & D^{i+1,j} & \\
 & & C^{i+1,j-1}\ar[u]^-{\dif_2} \ar[r]^-{\dif_1} & \ddots\\
}
\end{gathered}
\right)
 ,\\
\left(
\begin{gathered}
\xymatrix@=1.5em{
\ddots \ar[r]^-{\dif_1} & D^{i-1,j+1} & & \\
 & C^{i,j-1}\ar[r]^-{\dif_1}\ar[u]^-{\dif_2} & D^{i,j} & \\
 & & C^{i,j-1} \ar[u]^-{\dif_2} \ar[r]^-{\dif_1} & D^{i+1,j-1} \\
 & & & \ddots\ar[u]^-{\dif_2} \\
}
\end{gathered}
\right).
\end{gather*}

In particular, if we further assume that $N$ as above is indecomposable, then there must be $(i,j)\in \Z^2$ such that $N$ is isomorphic to one of the above ``zig-zag'' modules, and either $C^{i,j}=N^{i,j}$ or $D^{i,j}=N^{i,j}$. Flattening out the zig-zag, say, the first one, we may identify $N$ with an indecomposable finite-dimensional representation of the $A$ quiver with the alternating orientation
\[
\xymatrix{
\cdots & C^{i-1,j+1}\ar[l]_-{\dif_2} \ar[r]^-{\dif_1} & D^{i,j+1} & C^{i,j}\ar[l]_-{\dif_2}\ar[r]^-{\dif_1} & D^{i+1,j} & C^{i+1,j-1}\ar[l]_-{\dif_2}\ar[r]^-{\dif_1} & \cdots
} .
\]
By the classical result of Gabriel (see, for instance, \cite{Sk}), such an indecomposable module must be of the form
\[
\xymatrix{
\cdots & \Bbbk \ar[l]_-{=} \ar[r]^-{=} & \Bbbk & \Bbbk \ar[l]_-{=}\ar[r]^-{=} & \Bbbk & \Bbbk \ar[l]_-{=}\ar[r]^-{=} & \cdots
} .
\]
Such an indecomposable module translates back into either the simple module or a zig-zag module listed
in Example~\ref{egindecomposableLambda2mod} (cases (1), (3) and (4)). The theorem follows.

\begin{rem}[unbounded complexes] As the proof reveals above, one may extend Theorem~\ref{thmclassbicomplex} to the case of unbounded bicomplexes as well.

{\bf Case 5.}
Initially vertical and bounded from ``below'' or ``above''; the bounded corner sitting in bidegree $(i,j)$:
 \[
 Z^{i,j}_{\uparrow,+}\colon \
\begin{gathered}
\xymatrix{
 \ddots \ar[r] & \Bbbk & & \\
 & \ddots \ar[r]\ar[u] & \Bbbk & \\
 & & \Bbbk \ar[u] \ar[r] & \Bbbk\\
 & & & \Bbbk \ar[u]
}
\end{gathered}
 \qquad
Z^{i,j}_{\uparrow,-}\colon \
\begin{gathered}
\xymatrix{
 \Bbbk & & & \\
 \Bbbk\ar[u] \ar[r] & \Bbbk & & \\
 & \ddots \ar[r]\ar[u] & \Bbbk & \\
 & & \Bbbk \ar[u] \ar[r] & \ddots
}
\end{gathered}
\]
Cohomology spaces of $Z^{i,j}_{\uparrow,\pm}$ with respect to the vertical differential $\dif_2$ are both zero. But the collapsed total complexes, which both have the form
\[
0\lra \Bbbk^\infty \stackrel{d}{\lra} \Bbbk^\infty \lra 0,
\]
have different total cohomologies. It is readily seen that, for $Z^{i,j}_{\uparrow,+}$, the total differential is both injective and surjective. However, for $Z^{i,j}_{\uparrow,-}$, the total differential is injective, but not surjective. The cokernel of $d$ is given by $\Bbbk$ sitting in the bidegree $(i,j)$.

{\bf Case 6.}
The module $Z^{i,j}_{\rightarrow,\pm}$, which starts horizontally and is bounded from below or above, whose bounded corner lies in bidegree $(i,j)$:
\[
Z^{i,j}_{\rightarrow,+}\colon \
\begin{gathered}
\xymatrix{
 \ddots \ar[r] & \Bbbk & & &\\
 & \Bbbk\ar[u] \ar[r] & \Bbbk & & \\
 & & \ddots \ar[r]\ar[u] & \Bbbk &\\
 & & & \Bbbk \ar[u] \ar[r] & \Bbbk
}
\end{gathered}
\qquad
Z^{i,j}_{\rightarrow,-} \colon \
\begin{gathered}
\xymatrix{
 \Bbbk \ar[r] & \Bbbk & & &\\
 & \Bbbk\ar[u] \ar[r] & \Bbbk & & \\
 & & \ddots \ar[r]\ar[u] & \Bbbk &\\
 & & & \Bbbk \ar[u] \ar[r] & \ddots
}
\end{gathered}
\]
Cohomology spaces with respect to $\dif_2$ give a single $\Bbbk$ in
bidegree $(i,j)$. However, the total cohomology of the collapsed complexes
\begin{gather*} 0 \lra \Bbbk^{\infty} \stackrel{d}{\lra} \Bbbk^\infty \lra 0\end{gather*}
behaves differently. For $Z^{i,j}_{\rightarrow,+}$, the total differential is clearly injective, but not surjective. The cohomology classes represented by the vectors $1$ sitting in bidegrees $(i-r,j+r)$, $r\in \N$, are all cohomologous, and their images in the total complex represent the same cohomology class in degree $i+j$. On the other hand, the total differential of $Z^{i,j}_{\rightarrow,-}$ is an isomorphism, and thus there is no total cohomology.

{\bf Case 7.}
The module $Z^{i,j}_{\pm}$, which is unbounded in both directions. The underlined copy of~$\Bbbk$ sits in bidegree $(i,j)$. The modules are taken to be the same up to shifting $(i,j)$ to $(i+r,j-r)$, where $r\in \Z$, and identifying $Z^{i,j}_+$ with $Z_{-}^{i+1,j}$:
\[
Z^{i,j}_{+} \colon \
\begin{gathered}
\xymatrix{
 \ddots \ar[r] & \Bbbk & \\
 & \underline{\Bbbk}\ar[u] \ar[r] & \Bbbk \\
 & & \ddots \ar[u]
}
\end{gathered}
\qquad
Z^{i,j}_{-} \colon \
\begin{gathered}
\xymatrix{
 \ddots & & \\
 \Bbbk\ar[u] \ar[r] & \underline{\Bbbk} & \\
 & \Bbbk \ar[r]\ar[u] & \ddots \\
}
\end{gathered}
\]
Again, the vertical cohomology with respect to $\dif_2$ of $Z^{i,j}_{\pm }$ are both zero. The total cohomology for the collapsed complexes both have one-dimensional cohomology sitting in the cokernel of~$d$. In this case, the spectral sequences will not converge.
\end{rem}

Let us call a bicomplex $M=\oplus_{i,j\in \Z}M^{i,j}$ \emph{bounded from Southeast} when $M^{i,j}=0$ if $i\gg 0$ and $j \ll 0$. A bicomplex $M$ is called \emph{bounded from Northwest} when $M^{i,j}=0$ if $i \ll 0 $ and $j \gg 0$. Combining with the observations in Section~\ref{subsec-bicomplexes}, we see that if a bicomplex $M$ is bounded from Southeast, then, together with finite-dimensional summands, $M$ may contain additional summands of the form $Z^{i,j}_{\uparrow,+}$ and $Z^{i,j}_{\rightarrow,+}$. However, taking $\dif_2$-cohomology first does not create additional classes that need to be killed off in the total cohomology. Similarly, bicomplexes that are bounded from Northwest may contain infinite-dimensional summands of type~$Z^{i,j}_{\uparrow, -}$ and~$Z^{i,j}_{\rightarrow, -}$. Taking $\dif_1$ cohomology contributes nothing towards total cohomology.

\begin{cor}If $M$ is a bicomplex bounded from Southeast, then there is a spectral sequence whose $E_1$ page equals $(\mH(M,\dif_2),\dif_1)$, converging to the total cohomology of~$M$. Likewise, if~$M$ is a bicomplex bounded from Northwest, then there is a spectral sequence starting at $(\mH(M,\dif_1),\dif_2)$ converging to the total cohomology of~$M$.
\end{cor}

Let us call a complex semibounded if it bounded from either Northwest or Southeast. A semibounded complex cannot contain summands $Z^{i,j}_{+}$ or $Z^{i,j}_{-}$ that prevent either spectral sequence from converging.

\subsection{Connection to zig-zag algebras}\label{subsec-Lambda-2-zig-zag}
Let us point out the connection between the category $\mc{M}_2$ of bicomplexes with the module category over (an infinite version of) the zig-zag algebra considered in~\cite{KhSe}.

Let $Q_\infty$ be the following quiver whose vertices are labelled by $r\in \Z$:
\begin{gather}\label{quiver-Q}
\xymatrix{
 \cdots &
\overset{r-2}{~\circ~} \ltwocell{'}&
 \overset{r-1}{~\circ~}\ltwocell{'}&
 \overset{r}{\circ}\ltwocell{'}&
 \overset{r+1}{~\circ~}\ltwocell{'}&
 \overset{r+2}{~\circ~} \ltwocell{'}&
 \cdots \ltwocell{'}
 }
\end{gather}
Set $\Bbbk Q_\infty$ to be the path algebra associated to $Q_\infty$ over the ground field. We use, for instance, notation $(i|j|k)$, where $i$, $j$, $k$ are vertices of the quiver $Q_\infty$, to denote the path which starts at vertex~$i$, then goes through $j$ (necessarily $j=i\pm 1$) and ends at~$k$. The composition of paths is given by
\begin{gather*}
(i_i|i_2|\cdots|i_r)\cdot (j_1|j_2|\cdots|j_s)=
\begin{cases}
(i_i|i_2|\cdots|i_r|j_2|\cdots|j_s), & \textrm{if $i_r=j_1$,}\\
0, & \textrm{otherwise,}
\end{cases}
\end{gather*}
where $i_1,\dots, i_r$ and $j_1, \dots, j_s$ are sequences of neighboring vertices in $Q_\infty$.

\begin{defn}\label{def-algebra-A} The \emph{zig-zag algebra $A=A_\infty$} is the quotient of the path algebra $\Bbbk Q_\infty$ by the relations, for any $r\in \Z$,
\[
(r|r+1|r+2)=0,\qquad (r|r-1|r-2)=0,\qquad (r|r-1|r)=(r|r+1|r).
\]
\end{defn}

We make the zig-zag algebra graded by setting\footnote{The grading is chosen to match with the convention of \cite{KhSe}.}
\begin{gather*}
 \deg (r)=\deg (r|r+1)=0,\qquad \deg (r|r-1)=1,
\end{gather*}
for all $r\in \Z$. It is a non-unital algebra with a system of mutually orthogonal idempotents $\{(r)|r\in \Z\}$. There is an obvious automorphism~$T$ on~$A$, defined by
\begin{gather*}
 T(r):=(r+1),\qquad T(r|r+1):=(r+1|r+2), \qquad T(r|r-1):=(r+1|r).
\end{gather*}

For a fixed pair of integers $(r,i)\in \Z^2$, there is a graded projective module $ P_r\langle i\rangle$ which is generated by the idempotent $(r)$, whose degree is shifted up by~$ i $. More explicitly,
$P_{r}\langle i\rangle$ is the four-dimensional vector space with the basis
\begin{gather*}
\big\{ (r)\sigma_i,\,(r+1|r)\sigma_i,\, (r-1|r)\sigma_i, \,(r|r+1|r)\sigma_i\big\},
\end{gather*}
where $\sigma_i$ stands for the module generator sitting in degree~$i$.

We will consider the category of graded modules over $A$, which we denote by $\mc{M}(A)$, in what follows. The automorphism $T$ of $A$ induces an autoequivalence $\mc{T}$ of $\mc{M}(A)$, defined by~$\mc{T}:=\big(T^{-1}\big)^*$. Clearly $\mc{T}(P_r\langle i \rangle)= P_{r+1}\langle i \rangle$ holds for all $r,i\in \Z$.

Given a module $M=\oplus_{i,j\in \Z}M^{i,j}$ in $\mc{M}_2$, we place the homogeneous bigraded component of~$M^{i,j}$ at $(i,j)$ in the corresponding node of the two-dimensional lattice~$\Z^2$. For each $r\in \Z$, we collect together~$M^{i,j}$s on the line of slope one (depicted as the dashed line in the picture below):
\begin{gather*}
M_r:=\bigoplus_{i-j=r}M^{i,j}.
\end{gather*}
Note that $M_{r}$ is singly graded, with its homogeneous degree $j$ part $M_r^j$ set to be $M^{r+j,j}$.

Since $\dif_1$ and $\dif_2$ have bidegrees $(1,0)$ and $(0,1)$, respectively, they induce maps
\[
D_1:=\dif_1\colon \ M_{r}^j\lra M_{r+1}^j ,\qquad D_2:=(-1)^r\dif_2\colon \ M_{r}^j\lra M_{r-1}^{j+1},
\]
These maps satisfy $D_1^2=0$, $D_2^2=0$ and $D_1D_2=D_2D_1$.
We put the vector space $M_{r}$ at the $r$th vertex of $A$ and declare the rightward (resp.~leftward) going arrows to be the induced map $D_1$ (resp.~$D_2$).
We have thus obtained a graded $A$-module by summing over the $r$-degrees $M_\infty:=\oplus_{r\in \Z}M_{r}$.

Schematically, we depict the correspondence as follows:
\begin{gather*}
\begin{split}& \includegraphics[scale=0.79]{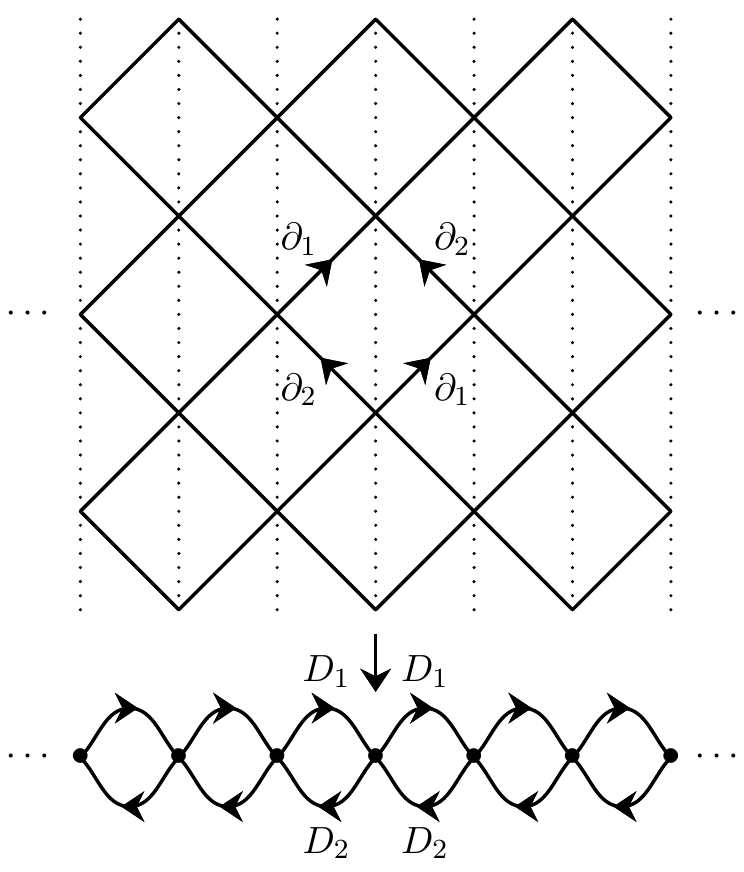}\end{split}
\end{gather*}
Furthermore, a morphism $f\colon M\lra N$ in $\mc{M}_2$ componentwise given by
\begin{gather*}f=\bigoplus_{i,j\in \Z}f^{i,j}\colon \ \bigoplus_{i,j\in \Z}M^{i,j}\lra \bigoplus_{i,j\in \Z}N^{i,j},\end{gather*}
satisfies
\[
\dif_1 f^{i,j}=f^{i+1,j}\dif_1, \qquad \dif_2 f^{i,j}=f^{i,j+1}\dif_2.
\]
One has the associated morphism of bigraded $A$-modules, which is defined as
\[
f_r:=\bigoplus_{i-j=r} f^{i,j} \colon \ M_r\lra N_r, \qquad f_\infty:=\bigoplus_{r\in \Z}f_r.
\]
Clearly $D_1f_r=f_{r+1}D_1$ and $D_2 f_r=f_{r-1}D_2$ holds for all $r\in \Z$, so that $f_\infty$ is a morphism of bigraded $A$-modules. This defines a functor $\mc{F}_\infty\colon \mc{M}_2\lra \mc{M}(A)$.

As the above functorial assignment is clearly reversible, the functor $\mc{F}_\infty$ is invertible.

\begin{prop}\label{prop-abelian-equivalence}
The functor $\mc{F}_\infty\colon \mc{M}_2\lra \mc{M}(A)$ is an equivalence of abelian categories. Furthermore, the functor satisfies
 \begin{gather*}\mc{F}_\infty (M\{1,0\})= \mc{T}(\mc{F}_\infty(M)), \\
 \mc{F}_\infty(M\{0,1\})=\mc{T}^{-1}(\mc{F}_{\infty}(M))\langle 1 \rangle
 .\end{gather*}
\end{prop}

\section{Tricomplexes and braid group actions} \label{sectricomplex}
\subsection{The monoidal category of tricomplexes}
We denote by $\mathbf{i} = (i_1,i_2,i_3)$ an ordered triple of integers, and write
$\mathbf{i}=i_1 e_1 + i_2 e_2 + i_3 e_3$ where
\begin{gather*} e_1=(1,0,0), \qquad e_2=(0,1,0),\qquad e_3=(0,0,1).\end{gather*}
In particular, we write $\mathbf{0}:=(0,0,0)$ as the additive unit element.

Let $\Lambda_3$ be the exterior algebra over $\Bbbk$ with three generators $\partial_1$, $\partial_2$, $\partial_3$:
\begin{gather*}
\partial_j^2 = 0, \qquad j=1,2,3, \qquad
\partial_j \partial_k + \partial_k \partial_j = 0, \qquad j\not= k.
\end{gather*}
We make $\Lambda_3$ a triply-graded $\Bbbk$-algebra, by assigning
degree $e_j$ to $\partial_j$. Let $\mc{M}_3$ be the category of triply-graded
left $\Lambda_3$-modules with respect to tri-degree preserving maps. A module $M$ consists of a collection of $\Bbbk$-vector spaces $M_{\mathbf{i}}$,
\begin{gather*} M = \bigoplus_{\mathbf{i}\in \Z^3} M_{\mathbf{i}},\end{gather*}
together \looseness=1 with linear maps $\partial_j \colon M_{\mathbf{i}} \lra M_{\mathbf{i}+e_j}$ subject
to the exterior algebra relations. It is useful to visualize $M$ as a 3-dimensional
object: the vector space $M_{\mathbf{i}}$ sits in the $\mathbf{i}$ node of a 3-dimen\-sional lattice and the maps~$\partial_j$ go along oriented edges of the lattice. Below is a portion of $M$ depicted:
\[
\xymatrix@!=2.5pc{
M_{(i,j,k+1)} \ar[dr]\ar[rr] && M_{(i,j+1,k+1)} \ar[rd] \ar[rr] && M_{(i,j+2,k+1)} \ar[rd]\\
& M_{(i+1,j,k+1)} \ar[rr] && M_{(i+1,j+1,k+1)} \ar[rr] && M_{(i+1,j+2,k+1)}\\
M_{(i,j,k)}\ar[uu]^{\dif_3} \ar'[r][rr]^(-0.25){\dif_2} \ar[dr]^{\dif_1} && M_{(i,j+1,k)} \ar'[u][uu] \ar[dr] \ar'[r][rr] && M_{(i,j+2,k)} \ar'[u][uu] \ar[dr]\\
 & M_{(i+1,j,k)} \ar[rr] \ar[uu] && M_{(i+1,j+1,k)} \ar[uu] \ar[rr] && M_{(i+1,j+2,k)} \ar[uu]\\
 M_{(i,j,k-1)}\ar[uu] \ar'[r][rr] \ar[dr] && M_{(i,j+1,k-1)} \ar'[u][uu] \ar[dr] \ar'[r][rr] && M_{(i,j+2,k-1)} \ar'[u][uu] \ar[dr]\\
 & M_{(i+1,j,k-1)} \ar[rr] \ar[uu] && M_{(i+1,j+1,k-1)} \ar[uu] \ar[rr] && M_{(i+1,j+2,k-1)} \ar[uu]
}
\]

The grading shift by $\mathbf{i}$, denoted $\{ \mathbf{i} \}$, is an automorphism of $\mc{M}_3$.
Any simple object of $\mc{M}_3$ is isomorphic to $S_{\mathbf{i}}:= \underline{\Bbbk}\{\mathbf{i}\}$
for a unique $\mathbf{i}$. Here $\underline{\Bbbk}$ is a one-dimensional $\Bbbk$-vector
space, in tridegree~$\mathbf{0}$, viewed as a $\Lambda_3$-module with the trivial action of $\partial_1$, $\partial_2$, $\partial_3$.

Any indecomposable projective in $\mc{M}_3$ is isomorphic to
$P_{\mathbf{i}}:= \Lambda_3\{\mathbf{i}\}$, for a unique $\mathbf{i}$. Any projective in $\mc{M}_3$ is
isomorphic to the direct sum of $P_{\mathbf{i}}$'s, possibly with infinite multiplicities.
Since $\Lambda_3$ is a trigraded Frobenius algebra, the $P_{\mathbf{i}}$ are also injective
objects of $\mc{M}_3$. A module $M$ contains $P_{\mathbf{i}}$ as a direct summand (and not
just as a submodule) if and only if $\partial_1 \partial_2 \partial_3 m \not= 0$ for some $m\in M_{\mathbf{i}}$.

Let $Q=\Lambda_3\omega /\Lambda_3 \partial_3 \omega$ be the cyclic module with one generator $\omega$ in tri-degree $\mathbf{0}$ and relation $\partial_3 \omega=0$. We depict $Q$ as a square
\begin{gather*}
\xymatrix{
\Bbbk \dif_2\omega \ar[r]^{\dif_1} & \Bbbk \dif_2\dif_1\omega\\
\Bbbk \omega \ar[u]^{\dif_2} \ar[r]^{\dif_1} & \Bbbk\dif_1\omega \ar[u]_{\dif_2}.
}
\end{gather*}
There is a graded isomorphism of modules $Q\cong \Lambda_3/\dif_3\Lambda_3$.

The algebra $\Lambda_3$ is a Hopf algebra in the category of trigraded (super) vector
spaces, where the (super) $\Z/2\Z$-grading is given by reducing $i_1+i_2+i_3$ modulo $2$, and $\Delta(\partial_r) = \partial_r \otimes 1 + 1 \otimes \partial_r $.
Consequently, the tensor product $M\otimes N$ of trigraded $\Lambda_3$-modules is a trigraded $\Lambda_3$-module, with $\partial_r$ acting by
\begin{gather*}\partial_r(m\otimes n) = \partial_r(m) \otimes n + (-1)^{i_r}m \otimes \partial_r(n), \qquad r=1,2,3,\end{gather*}
where $m$ is in degree $(i_1,i_2,i_3)$.

Similarly, there is a trigraded \emph{inner-hom} on $\mc{M}_3$, defined by
\begin{gather*}
 \HOM_\Bbbk(M,N):=\bigoplus_{\mathbf{i}\in \Z^3}\Hom_\Bbbk(M,N\{\mathbf{i}\}),
\end{gather*}
where the right hand side is the direct sum of homogeneous linear maps from $M$ to $N\{\mathbf{i}\}$. The inner hom space carries a natural $\Lambda_3$ action defined by, for any $f\in\Hom_\Bbbk(M,N\{i_1,i_2,i_3\})$
\begin{gather}\label{eqn-d-action-on-morphism}
\dif_r(f)(m)=\dif_r(f(m))-(-1)^{i_r}f(\dif_r(m)), \qquad r=1,2,3.
\end{gather}
The spaces of $\Lambda_3$-invariants under this action consist of morphisms in $\mc{M}_3$ of all degrees:
\begin{gather*}
\HOM_\Bbbk(M,N)^{\Lambda_3}=\bigoplus_{\mathbf{i}\in \Z^3} \Hom_{\mc{M}_3}(M,N\{\mathbf{i}\}).
\end{gather*}

It is useful to regard $\Lambda_2$ and $\Lambda_1$ as certain graded Hopf subalgebras in $\Lambda_3$. To do this, we break the apparent symmetry and define $\Lambda_2$ to be the subalgebra generated by $\dif_1$ and $\dif_2$, while setting $\Lambda_1^\prime$ to be the subalgebra generated by $\dif_3$. The natural algebra inclusions
\begin{gather*}
\iota\colon \ \Lambda_2\hookrightarrow \Lambda_3,\qquad \jmath\colon \ \Lambda_1^\prime \subset \Lambda_3
\end{gather*}
admit retractions
\begin{gather}\label{eqn-augmentations}
\mu\colon \ \Lambda_3\lra \Lambda_2, \qquad \nu\colon \ \Lambda_3\lra \Lambda_1^\prime,
\end{gather}
which are respectively given by setting $\dif_3$ or $\dif_1$, $\dif_2$ to be zero.

Using these subquotient algebras, we define a functor by taking ``partial graded-hom'' with respect to $\Lambda_1'$, as follows. Fix $i$ and $j$ degrees. Given any $M\in \mc{M}_3$, set
\begin{gather*}
M_{i,j}:= \nu^*\left( \bigoplus_{k\in \Z}M_{i,j,k}\right),
\end{gather*}
where in the last term, we only keep the $\Lambda_1^\prime$-module structure on $\oplus_{k} M_{i,j,k}$.
The functor extends naturally to morphisms in $\mc{M}_3$, and has the effect, on objects, of taking the direct sum of $M_{i,j,k}$ over $k\in \Z$. It remembers the $\dif_3$-complex structure inherited from that of~$M$, while making~$\partial_1$,~$\partial_2$ act by~$0$.

\subsection{A braid group action}
In this section, we exhibit a braid group action on the stable category of trigraded $\Lambda_3$-modules.

The tensor product $Q\otimes M_{i,j}$ is an object of $\mc{M}_3$, with
$\partial_1$, $\partial_2$ acting only along $Q$ (since their actions on $M_{i,j}$
are trivial) and $\partial_3$ acting along $M_{i,j}$.

Consider the functor
\begin{gather*}
\mc{U}_r(M) := \bigoplus_{i-j=r} Q\otimes M_{i,j} .
\end{gather*}

Geometrically, we take the plane $P_r=\{(i,j,k)\,|\, i-j=r\}$ in $\Z^3$,
with vector spaces $M_{\mathbf{i}}$ sitting in the nodes, and form four copies
of the plane (the tensor product with~$Q$) related by the differentials $\partial_1 $
and $\partial_2$. The differential $\partial_3$ acts along edges $(\mathbf{i},\mathbf{i}+e_3)$
contained in the plane~$P_r$. We depict the summand $Q\otimes M_{i,j}$ in the next diagram. For a fixed $e_3$-degree $k$, $Q\otimes M_{i,j,k}$ has four copies of $M_{i,j,k}$ sitting in degrees $(i,j,k)$, $(i+1,j,k)$, $(i,j+1,k)$ and $(i+1,j+1,k)$ respectively. They correspond to
\begin{gather*}
\Bbbk\omega\otimes M_{i,j,k},\qquad
\Bbbk\dif_1\omega\otimes M_{i,j,k},\qquad
\Bbbk\dif_2\omega\otimes M_{i,j,k}, \qquad
\Bbbk \dif_2\dif_1\omega \otimes M_{i,j,k}.
\end{gather*}
All maps except for{\samepage
\begin{gather*}\dif_1\colon \ M_{i,j,k}\cong \Bbbk\dif_2\omega \otimes M_{i,j,k}\lra \Bbbk \dif_2\dif_1\otimes M_{i,j,k} \cong M_{i,j,k}\end{gather*}
act as identity maps, which is the negative identity map.}

Now, summing over $k$ and keeping track of the differential $\dif_3$, we obtain the diagram
\begin{gather}\label{eqn-cube-diagram}
\begin{split}& \includegraphics[scale=0.8]{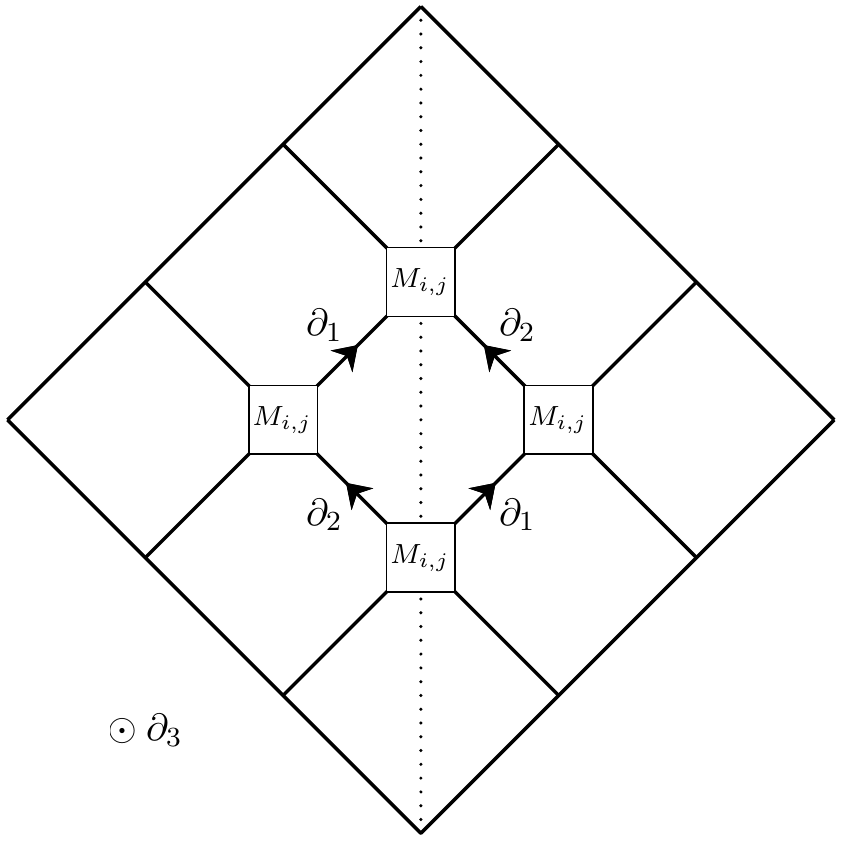}\end{split}
\end{gather}
Here the differential $\dif_3$ points perpendicularly out of the plane.

\begin{prop} \label{prop-TL-action}
The following isomorphisms between endofunctors of $\mc{M}_3$ hold:
\begin{gather*}
 \mc{U}_r^2 \cong \mc{U}_r\{1,1,0\} \oplus \mc{U}_r ,\\
 \mc{U}_r \mc{U}_{r\pm 1} \mc{U}_r \cong \mc{U}_r\{1,1,0\},\\
 \mc{U}_r \mc{U}_s = 0 \qquad \mathrm{if} \quad |r-s|>1.
\end{gather*}
\end{prop}
\begin{proof}
We start with the first equation. We compute the left hand side as
\begin{align*}
\mc{U}_r^2(M) & = \bigoplus_{i-j=r} \mc{U}_r(Q\otimes M_{i,j})
 = \bigoplus_{i-j=r} \left(\bigoplus_{k-l=r} Q\otimes (Q\otimes M_{i,j})_{k,l}\right) \nonumber \\
& = \bigoplus_{i-j=r} Q\otimes \left(\Bbbk\omega \otimes M_{i,j}\oplus \Bbbk\dif_2\dif_1\omega\otimes M_{i,j}\right)\nonumber\\
& \cong \left(\bigoplus_{i-j=r} Q\otimes \Bbbk\omega \otimes M_{i,j}\right)\oplus
 \left(\bigoplus_{i-j=r} Q\otimes \Bbbk\dif_2\dif_1\omega \otimes M_{i,j}\right) \nonumber\\
& \cong \mc{U}_r(M)\oplus \mc{U}_r(M)\{1,1,0\}.
\end{align*}
Here, in the third equality, we have used that $Q\otimes M_{i,j}$ has only two terms concentrated on the line $k-l=r$ (see the above picture~\eqref{eqn-cube-diagram}).

For the second isomorphism, we have (taking the $r+1$ case)
\begin{align}
\mc{U}_r\mc{U}_{r+1}\mc{U}_r(M) & = \bigoplus_{i-j=r} \mc{U}_r\mc{U}_{r+1}(Q\otimes M_{i,j})
 = \bigoplus_{i-j=r}\mc{U}_r \left(Q \otimes (\Bbbk\dif_1\omega\otimes M_{i,j})\right)\nonumber\\
& = \bigoplus_{i-j=r} Q\otimes \Bbbk\dif_2\omega \otimes \Bbbk\dif_1\omega \otimes M_{i,j}
 = \bigoplus_{i-j=r} Q\otimes M_{i,j}\{1,1,0\} \nonumber\\
& = \mc{U}_r(M)\{1,1,0\}.\label{eqnproofofTL2}
\end{align}

The last isomorphism is easy, and we leave it as an exercise to the reader.
\end{proof}

\begin{rem}Perhaps the cartoon below, in the scheme of equation~\eqref{eqn-cube-diagram}, helps visualizing the equalities in the above proof. We show this for equation~\eqref{eqnproofofTL2} as an example. Depict a copy of~$M_{i,j}$ by a box in the lattices below. A black dot in a box indicates the term contributing to the functor on the outward arrow:
\begin{gather*}
\includegraphics{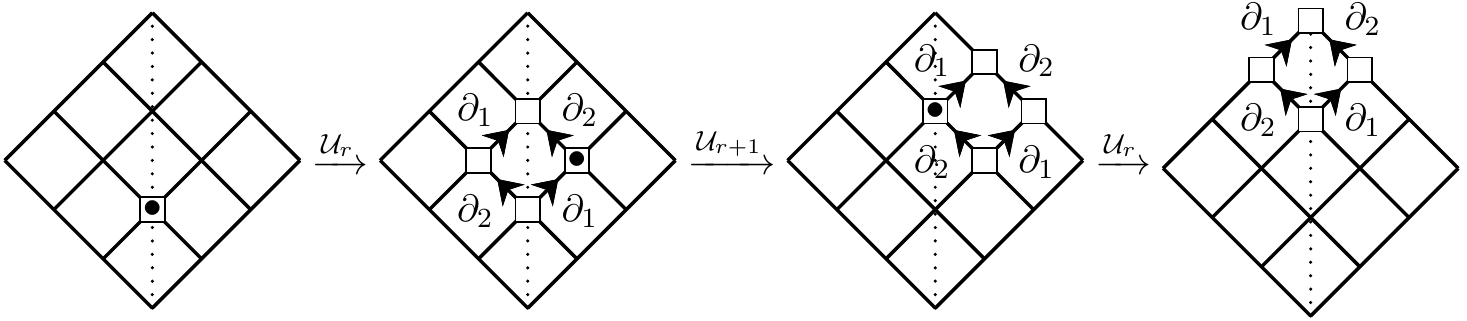}
\end{gather*}
\end{rem}

There exists a unique morphism in $\mc{M}_3$
\begin{gather}\label{mor-1}
 Q \otimes M_{i,j} \lra M,
\end{gather}
which takes $ \omega\otimes m $ to $m$. This morphism takes
$\partial_1 \omega\otimes m$ to $\partial_1 m$, etc.

Summing over $i$, $j$ such that $i-j=r$, morphisms~(\ref{mor-1}) combine into a module homomorphism
\begin{gather*}
\mathrm{in}_r\colon \ \mc{U}_r(M) \lra M
\end{gather*}
natural in $M$. Thus, $\mathrm{in}_r\colon \mc{U}_r \Longrightarrow \mathrm{Id}$ is a natural transformation of functors on~$\mc{M}_3$.

Next, we construct a module homomorphism
\begin{gather*}
 M \xrightarrow{\mathrm{out}_r} \mc{U}_r(M) \{-1,-1,0\} .
\end{gather*}

Denote by $M^\nu$ the underlying trigraded vector space of $M$, while only remembering the $\Lambda_1^\prime$-module structure. Consider the map
\begin{align*}
 \mathrm{out}\colon M & \lra Q\otimes M^{\nu}\{-1,-1,0\}, \\
 m & \mapsto (-1)^{i+j} (\omega \otimes \dif_1\dif_2(m)+\dif_1 \omega \otimes \dif_2(m)- \dif_2 \omega \otimes \dif_1(m)+ \dif_1\dif_2\omega \otimes m ),
\end{align*}
where $m\in M_{i,j,k}$ is a homogeneous element.

\begin{lem}The map $\mathrm{out}\colon M\lra Q\otimes M^\nu \{-1,-1,0\}$ is a morphism of trigraded $\Lambda_3$-modules.
\end{lem}

\begin{proof}The map clearly commutes with $\dif_3$-actions on both sides, as $\dif_3$ kills $\omega$ and anti-commutes with $\dif_1$ and $\dif_2$. To verify that $\mathrm{out}$ also commutes with $\dif_1$ and $\dif_2$ requires a small computation. We check, for instance, that it commutes with $\dif_1$, and leave the $\dif_2$-computation to the reader.

On the one hand, if $m\in M_{i,j,k}$, and using that $\dif_1$ acts trivially on $M^\nu$, we have
\begin{gather*}
\dif_1(\mathrm{out}(m)) = (-1)^{i+j} (\dif_1\omega\otimes \dif_1\dif_2(m)-\dif_1\dif_2\omega\otimes \dif_1 (m)).
\end{gather*}
On the other hand,
\begin{align*}
\mathrm{out}(\dif_1(m))& = (-1)^{i+j+1}\left( \dif_1\omega \otimes \dif_2(\dif_1(m))+\dif_1\dif_2\omega\otimes \dif_1(m)\right) \nonumber\\
& = (-1)^{i+j+1}\left( - \dif_1\omega \otimes \dif_1\dif_2(m)+\dif_1\dif_2\omega\otimes \dif_1(m)\right) .
\end{align*}
Comparing these expressions, the commutativity with the $\dif_1$-actions follows.
\end{proof}

Since $Q\otimes M^\nu$ naturally decomposes into a direct sum of $\Lambda_3$-modules
\[
Q\otimes M^\nu\cong \bigoplus_{i,j\in \Z} Q\otimes M_{i,j},
\]
for each $r\in \Z$, we have a natural projection map of $\Lambda_3$-modules
\[
\pi_r\colon \ Q\otimes M\lra \bigoplus_{i-j=r} Q\otimes M_{i,j}.
\]
We can thus define the composition map
\begin{gather*}
\mathrm{out}_r:=\pi_r\circ\mathrm{out}\colon \ M\lra \mc{U}_r(M)\{-1,-1,0\}.
\end{gather*}

Componentwise, $\mathrm{out}_r$ has the effect, for a homogeneous $m\in M_{i,j,k}$,
\begin{gather}\label{eqn-out-componentwise}
\mathrm{out}_r(m):=
\begin{cases}
(-1)^{i+j} (\omega \otimes \dif_1\dif_2 m+ \dif_1\dif_2\omega \otimes m), & \textrm{if}~i-j=r,\\
(-1)^{i+j}\dif_1\omega \otimes \dif_2(m), & \textrm{if}~i-j=r+1, \\
(-1)^{i+j+1} \dif_2\omega \otimes \dif_1(m), & \textrm{if}~i-j=r-1,\\
0, & \textrm{otherwise}.
\end{cases}
\end{gather}
We have thus obtained $\mathrm{out}_r$ as a tri-grading preserving homomorphism of $\Lambda_3$-modules, functorial in~$M$. In other words, similarly as for $\mathrm{in}_r$, the map $\mathrm{out}_r\colon \Id\Rightarrow \mc{U}_r\{-1,-1,0\}$ is a natural transformation of functors.

Let $\mc{SM}_3$ be the stable category of trigraded left $\Lambda_3$-modules.
It has the same objects as $\mc{M}_3$ and the morphisms are those in
$\mc{M}_3$ modulo morphisms that factor through a projective object of $\mc{M}_3$.
In particular, a projective trigraded $\Lambda_3$-module is isomorphic to the
zero object in $\mc{SM}_3$. The stable category is triangulated, with the shift
functor $[1]_{\mc{SM}}$ taking $M$ to the cokernel of an inclusion $M\subset P$, where
$P$ is a projective module. For concreteness, we can choose $P$ to be $\Lambda_3\otimes M
\{-1,-1,-1\}$, with the inclusion taking $m$ to $\partial_1\partial_2\partial_3 \otimes m$.
The shift by $\{-1,-1,-1\}$ makes the inclusion grading-preserving.
Then $M[1]_{\mc{SM}}= \widehat{\Lambda}\otimes M$ where
\begin{gather*}\widehat{\Lambda} = \Lambda_3/(\partial_1\partial_2\partial_3) \{-1,-1,-1\}.\end{gather*}
The cone of a morphism $f\colon M \lra N$ is defined as the cokernel of the inclusion
\begin{gather*} M \subset N \oplus (\Lambda_3 \otimes M\{-1,-1,-1\}),\end{gather*}
which takes $m$ to $(f(m), \partial_1\partial_2\partial_3 (m))$. For more details, we refer the reader to Happel \cite{Ha}. We will need the following result computing morphism spaces in $\mc{SM}_3$, bearing in mind the $\Lambda_3$ action defined in equation~\eqref{eqn-d-action-on-morphism}.

\begin{lem}\label{lem-morphism-in-stable-cat}
Given two objects $M,N\in \mc{SM}_3$, there is an isomorphism
\[
\Hom_{\mc{SM}_3}(M,N)=\dfrac{\Hom_{\mc{M}_3}(M,N)}{\dif_1\dif_2\dif_3 \Hom_{\Bbbk}(M,N\{-1,-1,-1\})}.
\]
\end{lem}
\begin{proof}
See \cite[Corollary 5.5]{Q}.
\end{proof}

We introduce another cone construction defined for morphisms in the abelian category $\mc{M}_3$.
Given a morphism $f\colon M\lra N$ in $\mc{M}_3$, the $\dif_3$-cone $\mathrm{C}_3(f)$, as a trigraded vector space, is the object $M\{0,0,-1\} \oplus N$, on which the $\Lambda_3$-generators act by
\begin{gather}\label{eqn-d3-action-on-d3-cone}
 \partial_3(m,n)=(-\partial_3 m , f(m)+\partial_3(n)),
\end{gather}
and $\partial_j(m,n)=(\partial_j m, \partial_j n)$ for $j=1,2$.

Alternatively, regard $\Lambda_1^\prime=\Bbbk[\dif_3]/(\dif_3^2)$ as a trigraded $\Lambda_3$-module via the homomorphism $\nu$ (see equation \eqref{eqn-augmentations}), the $\dif_3$-cone is defined as the push-out of $f\colon M\lra N$ and $\dif_3\otimes \Id_M \colon M \lra \Lambda_1^\prime \otimes M$. This is the top square of the following diagram, whose columns are short exact in the abelian category because of the push-out property:
\begin{gather}\label{eqn-C3cone-as-pushout}
\begin{split}&
\xymatrix{
M \ar[r]^f \ar[d]_{\dif_3\otimes \Id_M} & N \ar[d]\\
\Lambda_1^\prime \otimes M\{0,0,-1\} \ar[d] \ar[r] & \mathrm{C}_3(f) \ar[d]\\
M\{0,0,-1\} \ar@{=}[r] & M\{0,0,-1\}.
}
\end{split}
\end{gather}

\looseness=-1 Define $\mc{R}_r:=\mathrm{C}_3(\mathrm{in}_r)$, i.e., it is the functor in $\mc{M}_3$ that takes a module $M$ to the $\dif_3$-cone of the homomorphism $\mathrm{in}_r\colon \mc{U}_r(M) \lra M$. Let $\mc{R}_r^\prime:=\mathrm{C}_3(\mathrm{out}_r)\{0,0,1\}$, which takes $M$ to the $\dif_3$-cone of $\mathrm{out}_r\colon M \lra \mc{U}_r(M)\{-1,-1,0\}$, with the grading
shifted by $\{0,0,1\}$, so that the vector spaces in the nodes of $M$ stay
in their original tridegrees, and $\partial_3$ changes sign in its action on
$\mc{U}_r(M)$, not~$M$.

\begin{lem}The functors $\mc{R}_r$, $\mc{R}_r^\prime$ descend to well-defined functors on the stable category $\mc{SM}_3$.
\end{lem}
\begin{proof}It suffices to show that, if $M$ is a projective $\Lambda_3$-module, then $\mc{R}_r(M)$ and $\mc{R}^\prime_r(M)$ are both projective. Let us do this for $\mc{R}_r$, and the $\mc{R}^\prime_r$ case is similar.

By \eqref{eqn-C3cone-as-pushout}, $\mc{R}_r(M)$ fits into a short exact sequence of $\Lambda_3$-modules
\[
0\lra M \lra \mc{R}_r(M)\lra \mc{U}_r(M)\{0,0,-1\} \lra 0.
\]
Since $\Lambda_3$ is Frobenius, $M$ is also injective and the above sequence splits. We are thus reduced to showing that $\mc{U}_r(M)\{0,0,-1\}$ is graded projective. Without loss of generality, we may assume that $M\cong \Lambda_3\{i,j,k\}$ is indecomposable. As $\Lambda_1^\prime$-modules, there is a direct sum decompostion
\[
\Lambda_3\cong \Lambda_1^\prime\oplus \Lambda_1^\prime\{1,0,0\} \oplus \Lambda_1^\prime\{0,1,0\} \oplus \Lambda_1^\prime\{1,1,0\}.
\]
Using this decomposition and the fact that $Q\otimes \nu^*(\Lambda_1^\prime) \cong \Lambda_3$, we have
\begin{gather*}
\mc{U}_r(\Lambda_3\{i,j,k\})
=\begin{cases}
\Lambda_3\{i,j,k\} \oplus \Lambda_3\{i+1,j+1,k\}, & i-j=r,\\
\Lambda_3\{i+1,j,k\}, & i-j=r+1,\\
\Lambda_3\{i,j+1,k\}, & i-j=r-1,\\
0, & |i-j-r|>1.
\end{cases}
\end{gather*}
The result follows.
\end{proof}

\begin{thm} \label{thm-3-braid-group-actions}\quad
\begin{enumerate}\itemsep=0pt
\item [$(i)$]The functors $\mc{R}_r,\mc{R}^\prime_r$ are invertible mutually-inverse endofunctors on the stable catego\-ry~$\mc{SM}_3$.
\item[$(ii)$]The following functor isomorphisms hold:
\begin{subequations}
\begin{gather}
\label{rel-yb}
\mc{R}_r \mc{R}_{r+1} \mc{R}_r \cong \mc{R}_{r+1} \mc{R}_r \mc{R}_{r+1},\\
\label{rel-comm}
\mc{R}_r \mc{R}_s \cong \mc{R}_s \mc{R}_r\qquad \mathrm{if}\quad |r-s|>1.
\end{gather}
\end{subequations}
\end{enumerate}
Consequently, the collection of functors $\{\mc{R}_r\,|\,r\in \Z\}$ gives rise to an action of the infinite braid group of infinitely many strands $\mathrm{Br}_\infty$ on the triangulated category $\mc{SM}_3$.
\end{thm}

The proof of the theorem will occupy the next subsection.

\begin{rem}In this section, we have interpreted the three differentials of~$\Lambda_3$ in two different ways: the $\Lambda_2\subset \Lambda_3$ plays the role of the algebra~$A$ (cf.~Section~\ref{subsec-Lambda-2-zig-zag}), while $\dif_3$ behaves more like a ``homological differential''. This apparent symmetry breaking allows one to construct three equivalent braid group actions on $\mc{SM}_3$ as in Theorem~\ref{thm-3-braid-group-actions}, by the automorphism of $\Lambda_3$ permuting the indices $\{1,2,3\}$.
\end{rem}

\subsection{Proof of Theorem \ref{thm-3-braid-group-actions}}
{\bf Invertibility of $\mc{R}_r$.} First we show $\mc{R}^\prime_r \mc{R}_r \cong \Id$. We check the effect of the left hand side on a trigraded $\Lambda_3$-module $M$.
\begin{align}\label{eqn-R-prime-R}
\mc{R}_r^\prime (\mc{R}_r(M)) & \cong \mc{R}_r^\prime \left(\mc{U}_r(M)\{0,0,-1\}\xrightarrow{\mathrm{in}_r}M \right)\nonumber \\
& \cong
\left(
\begin{gathered}
 \xymatrix{
 \mc{U}_r(M)\{0,0,-1\} \ar[r]^-{-\mathrm{in}_r}\ar[d]_{\mathrm{out}_r} & M \ar[d]^{\mathrm{out}_r} \\
 \mc{U}_r^2(M)\{-1,-1,0\} \ar[r]^-{\mathrm{in}_r} & \mc{U}_r(M)\{-1,-1,1\}
 }
\end{gathered}
\right).
\end{align}
Here, in the diagram, the horizontal arrows are interpreted as the $\dif_3$-differential arising from the $\dif_3$-cone of $\mathrm{in}_r$, while the vertical arrows indicate that of $\mathrm{out}_r$. The differential action by~$\dif_1$,~$\dif_2$ preserves the position of the node, while the~$\dif_3$ acts both internally at the nodes and transfer elements long the arrows (see equation~\eqref{eqn-d3-action-on-d3-cone}).

By Proposition~\ref{prop-TL-action}, we may decompose
\begin{subequations}
\begin{gather}\label{eqn-U-square-term-1}
\mc{U}^2_r(M)\{-1,-1,0\}\cong
\mc{U}_r(M) \oplus \mc{U}_r(M)\{-1,-1,0\}.
\end{gather}
As in the proof of the proposition, we further identify
\begin{gather}\label{eqn-U-square-term-2}
\mc{U}_r(M)\cong \bigg(\bigoplus_{i-j=r} Q\otimes \dif_1\dif_2\omega \otimes M_{i,j} \bigg) \{-1,-1,0\},
\\
\label{eqn-U-square-term-3}
\mc{U}_r(M)\{-1,-1,0\}\cong \bigg(\bigoplus_{i-j=r} Q\otimes \omega \otimes M_{i,j} \bigg) \{-1,-1,0\}.
\end{gather}
\end{subequations}

By the definition of the $\dif_3$-cone, the sum of terms on the lower horizontal line of \eqref{eqn-R-prime-R} consitutes a $\Lambda_3$-submodule of $\mc{R}^\prime_r(\mc{R}_r(M))$. The morphism $\mathrm{in}_r$ on the lower horizontal line of~\eqref{eqn-R-prime-R} maps the summand \eqref{eqn-U-square-term-3} isomorphically onto $\mc{U}_r\{-1,-1,1\}$. Hence we have in~$\mc{R}^\prime_r(\mc{R}_r(M))$ a~$\Lambda_3$-submodule
\begin{gather}
\big(\mc{U}_r(M)\{-1,-1,0\}\xrightarrow{\mathrm{in}_r} \mc{U}_r(M)\{-1,-1,1\}\big)\nonumber\\
\qquad{} \cong \mc{U}_r(M)\otimes \nu^*(\Lambda_1^\prime) \cong \bigoplus_{i-j=r} (Q\otimes \nu^*(\Lambda_1^\prime)) \otimes M_{i,j}.\label{eqn-d3-cone-of-iso-proj}
\end{gather}
As $Q\otimes \Lambda_1^\prime \cong \Lambda_3$ is a tri-graded free $\Lambda_3$-module, it is a not only a submodule in $\mc{R}^\prime_r(\mc{R}_r(M))$ but also a direct summand, which is annihilated when passing to the stable category $\mc{SM}_3$. We thus may safely identify $ \mc{R}^\prime_r(\mc{R}_r(M))$ with the quotient of it by this submodule, which we denote by~$M_1$.

Now $M$ is clearly a $\Lambda_3$-submodule in $M_1$. We claim that $M_1/M$ is also a free $\Lambda_3$-module, and hence is a direct summand in $M_1$ whose complement is isomorphic to $M$.
It then follows that the natural inclusion map $M\hookrightarrow M_1$ is an isomorphism in $\mc{SM}_3$.

To prove the claim, note that
\begin{gather*}
 M_1/M\cong \big( \mc{U}_r(M)\{0,0,-1\}\xrightarrow{\mathrm{out}_r}\mc{U}_r(M) \big),
\end{gather*}
where the right hand side denotes a $\dif_3$-cone. If $m\in M_{i,j}$ is a homogeneous element, the map $\mathrm{out}_r$ has, by equation \eqref{eqn-out-componentwise}, the effect
\begin{gather*}
\mathrm{out}_r(\omega\otimes m) =(-1)^{i+j}\left(\omega\otimes \dif_1\dif_2\omega \otimes m + \dif_1\dif_2\omega\otimes \omega\otimes m\right), \\
\mathrm{out}_r(\dif_1 \omega\otimes m) = (-1)^{i+j+1} \dif_1\omega \otimes \dif_2\dif_1\omega \otimes m,\\
\mathrm{out}_r(\dif_2\omega\otimes m) = (-1)^{i+j+2} \dif_2\omega \otimes \dif_1\dif_2\omega \otimes m,\\
\mathrm{out}_r(\dif_1\dif_2\omega\otimes m) = (-1)^{i+j+2}\dif_1\dif_2\omega\otimes \dif_1\dif_2\omega \otimes m.
\end{gather*}
The right hand side of the first equation contains elements in $\mc{U}_r(M)\{-1,-1,0\}$ (see equation~\eqref{eqn-U-square-term-3}), which has already been mod out in~$M_1$. The rest of the terms on the right hand side of the equations have their middle term $\dif_1\dif_2\omega$. It follows that $\mathrm{out}_r$ maps $\mc{U}_r(M)\{0,0,-1\}$ isomorphically onto~$\mc{U}_r(M)$. The claim follows.

It is not hard to find a summand in $\mc{R}_r^\prime(\mc{R}_r(M))$ isomorphic to~$M$. Denote by
\begin{gather*}\mathrm{out}_r^{13}\colon \ M \lra Q\otimes Q^\nu \otimes M^\nu,
\qquad
m \mapsto \sum_i h_i \otimes \omega \otimes m_i,
\end{gather*}
where $h_i$, $m_i$ are the components of $\mathrm{out}_r(m)=\sum_i h_i\otimes m_i \in Q\otimes M^\nu$ as in equation~\eqref{eqn-out-componentwise}. The submodule
\[
\big\{\big({-}\mathrm{out}_r^{13}(m),m\big)\,|\,m\in M\big\}\subset\mc{U}_r^2(M)\{-1,-1,0\}\oplus M
\]
constitutes, by the above discussion, a $\Lambda_3$-summand isomorphic to $M$. The inclusion of this summand realizes the functor isomorphism $\mathrm{Id}\cong \mc{R}_r^\prime \mc{R}_r$.

The isomorphism $\mc{R}_r\mc{R}_r^\prime\cong \mathrm{Id}$ follows by a similar argument. Essentially, one just needs to flip the last term of~\eqref{eqn-R-prime-R} along its northwest-southeast diagonal. We leave the details to the reader as an exercise.

{\bf Braid relations.}
We next check the functor relations~\eqref{rel-yb} and~\eqref{rel-comm}.

The commutation relation~\eqref{rel-comm} is easy to check, as one can readily see that both sides are functorially isomorphic, when applied to a trigraded $\Lambda_3$-module $M$, to the $\dif_3$-cone of the morphism $\mathrm{in}_r\oplus \mathrm{in}_s$:
\begin{gather*}
\begin{gathered}
\xymatrix{
\mc{U}_r(M)\{0,0,-1\}\ar[dr]^{\mathrm{in}_r} & \\
& M.\\
\mc{U}_s(M)\{0,0,-1\}\ar[ur]_{\mathrm{in}_s} &
}
\end{gathered}
\end{gather*}
Here we have applied Proposition \ref{prop-TL-action} so that $\mc{U}_r\mc{U}_s(M)\cong 0 \cong \mc{U}_s\mc{U}_r(M)$.

To check the functor relation \eqref{rel-yb}, we first compute $\mc{R}_r\mc{R}_{r+1}\mc{R}_r$ applied to a $\Lambda_3$-module~$M$, which is equal to the total $\dif_3$-complex
\begin{gather*}
\begin{gathered}
\xymatrix@C=1.8em{
& \mc{U}_r\mc{U}_{r+1}(M)\{0,0,-2\}\ar[r]^-{-\mathrm{in}_{r+1}}\ar[dr]_<<{\mathrm{in}_r} & \mc{U}_r(M)\{0,0,-1\} \ar[drr]^-{\mathrm{in}_r} && \\
\mc{U}_r\mc{U}_{r+1}\mc{U}_r(M)\{0,0,-3\} \ar[ur]^-{\mathrm{in}_r} \ar[r]^-{-\mathrm{in}_{r+1}} \ar[dr]_-{\mathrm{in}_r} & \mc{U}_r^2(M)\{0,0,-2\} \ar[ur]^<<{-\mathrm{in}_r} \ar[dr]_<<{\mathrm{in}_r}& \mc{U}_{r+1}(M)\{0,0,-1\} \ar[rr]^-{\mathrm{in}_{r+1}}&& M.\\
& \mc{U}_{r+1}\mc{U}_r(M)\{0,0,-2\}\ar[r]_-{\mathrm{in}_{r+1}} \ar[ur]^<<{-\mathrm{in}_r} & \mc{U}_r(M)\{0,0,-1\} \ar[urr]_{\mathrm{in}_{r}} &&
}
\end{gathered}
\end{gather*}
We will gradually strip off the projective-injective summands of this module, which, for brevity, we will call~$M_0$ in what follows.

By Proposition \ref{prop-TL-action}, we identify
\begin{gather*}
 \mc{U}_r\mc{U}_{r+1}\mc{U}_r(M)\{0,0,-3\}\cong \mc{U}_r(M)\{1,1,-3\}\cong \bigoplus_{i-j=r}Q\otimes \dif_2\omega\otimes \dif_1\omega \otimes M_{i.j}\{0,0,-3\} .
\end{gather*}
By the definition \eqref{mor-1} of $\mathrm{in}_{r+1}$, the external $\dif_3$-differential $-\mathrm{in}_{r+1}$ maps this term isomorphically onto the summand $\mc{U}_r(M)\{1,1,-2\}$ of
\begin{gather}\label{eqn-U-square-in-yb}
\mc{U}_r^2(M)\{0,0,-2\}\cong \mc{U}_r(M)\{1,1,-2\}\oplus \mc{U}_r\{0,0,-2\}.
\end{gather}
Indeed, componentwise, the morphism has the effect
\begin{gather*}
Q\otimes \dif_2\omega\otimes \dif_1\omega \otimes M_{i.j}\{0,0,-3\} \lra Q\otimes \dif_1\dif_2\omega \otimes M_{i,j}\{0,0,-3\},\\
h\otimes \dif_2\omega\otimes \dif_1\omega\otimes m \mapsto -h\otimes \dif_2\dif_1\omega \otimes m.
\end{gather*}
Via these identifications, denote the direct sum of every term in $M_0$ other than $ \mc{U}_r(M)\{1,1,-3\}
$ and $\mc{U}_r(M)\{1,1,-2\}$ by $M_1$. Clearly $M_1$ is a $\Lambda_3$-submodule, whose quotient is equal to the $\dif_3$-cone
\[
\Big(\mc{U}_r(M)\{1,1,-3\}\xrightarrow{-\mathrm{Id}_{\mc{U}_r(M)\{1,1,-2\}}} \mc{U}_r(M)\{1,1,-2\}\Big).
\]
This cone is a projective-injective object in $\mc{M}_3$ (cf.~equation \eqref{eqn-d3-cone-of-iso-proj}). Hence~$M_1$ is isomorphic to $M_0$ in~$\mc{SM}_3$.

Next, inside $M_1$, the second summand of \eqref{eqn-U-square-in-yb} maps onto the anti-diagonal in the direct sum of two copies of $\mc{U}_r(M)\{0,0,-1\}$. Therefore
\[
\Big(\mc{U}_r(M)\{0,0,-2\}\xrightarrow{-\mathrm{in}_r\oplus \mathrm{in}_r} \mathrm{Im}(-\mathrm{in}_r\oplus \mathrm{in}_r)\Big)\cong \mathrm{C}_3 (\mathrm{Id}_{\mc{U}_r(M)\{0,0,-1\}} )
\]
is a projective $\Lambda_3$-module, and thus is isomorphic to zero in $\mc{SM}_3$. Modulo these terms, and equating the quotient as under the sum map
\begin{gather*}\mc{U}_r(M)\{0,0,-1\} \oplus \mc{U}_r(M)\{0,0,-1\}/\mathrm{Im}(-\mathrm{in}_r\oplus \mathrm{in}_r) \cong \mc{U}_r(M)\{0,0,-1\},\end{gather*}
we have that $M_1$ is isomorphic to the following total $\dif_3$-cone~$M_2$:
\begin{gather}\label{eqn-braid-complex-2}
M_2=\left(
\begin{gathered}
\xymatrix{
 \mc{U}_r\mc{U}_{r+1}(M)\{0,0,-2\}\ar[r]^-{\mathrm{in}_r}\ar[ddr]_<<{-\mathrm{in}_{r+1}} & \mc{U}_{r+1}(M)\{0,0,-1\} \ar[dr]^{\mathrm{in}_{r+1}} & \\
& & M\\
 \mc{U}_{r+1}\mc{U}_r(M)\{0,0,-2\}\ar[r]^-{\mathrm{in}_{r+1}} \ar[uur]^<<{-\mathrm{in}_r} & \mc{U}_r(M)\{0,0,-1\} \ar[ur]_{\mathrm{in}_r} &
}
\end{gathered}
 \right).
\end{gather}

It follows by the above discussion that $\mc{R}_r\mc{R}_{r+1}\mc{R}_r(M)$ is isomorphic to $M_2$ in $\mc{SM}_3$, and this isomorphism is clearly functorial in $M$.

A similar computation for $\mc{R}_{r+1}\mc{R}_{r}\mc{R}_{r+1}(M)$ shows that it is functorially isomorphic to~$M_2$ of equation~\eqref{eqn-braid-complex-2}. The braid relation follows.

\subsection{Connection to homological algebra of zig-zag algebras}
In view of Section~\ref{subsec-Lambda-2-zig-zag}, it is not surprising that trigraded $\Lambda_3$-modules are closely related to the homological algebra of the zig-zag algebra $A$. The main goal of this subsection is to utilize this relationship to establish the faithfulness of the braid group $\mathrm{Br}_\infty$ action on~$\mc{SM}_3$ (Theorem~\ref{thm-3-braid-group-actions}), building on the results of~\cite{KhSe}.

Let $M$ be a complex of graded $A$-modules
\[
M=\Big(
\cdots \lra M_{k-1} \stackrel{d_{k-1}}{\lra} M_k \stackrel{d_k}{\lra} M_{k+1} \lra \cdots
\Big),
\]
where each $M_k=\oplus_{j\in \Z}M_{k}^j$ is a graded $A$-module.
Recall that a morphism $f\colon M\lra N$ is called \emph{null-homotopic} if there is a collection of homogeneous $A$-module maps $h_k\colon M_{k}\lra N_{k-1}$, $k\in \Z$, such that $d_{k-1}h_{k}+h_{k+1}d_k=f_k$ holds for all~$k$. The homotopy category $\mc{C}(A)$ is the quotient of the category of chain complexes of graded $A$-modules by the ideal of null-homotopic morphisms. A complex $M$ is called \emph{contractible} if $\mathrm{Id}_M$ is null-homotopic.

The homotopy category $\mc{C}(A)$ carries two commuting grading shifts denoted by $\langle 1 \rangle$ and $[1]$ respectively. They are defined by
\begin{gather*}
 ( M\langle 1 \rangle )^j_k:=M^{j-1}_k,\qquad (M[1] )^j_k:=M^{j}_{k+1}.
\end{gather*}
In addition, the automorphism $\mc{T}$ of $\mc{M}(A)$ extends to an automorphism of $\mc{C}(A)$, denoted by the same letter, defined by termwise applying $\mc{T}$ on complexes:
\[
\mc{T}(M):=\Big(
\cdots \lra \mc{T}(M_{k-1}) \xrightarrow{\mc{T}(d_{k-1})} \mc{T}(M_k) \xrightarrow{\mc{T}(d_k)} \mc{T}(M_{k+1}) \lra \cdots
\Big).
\]

In what follows, we will also use the notation $\mc{C}(A\pmod)$ to stand for the full subcategory of~$\mc{C}(A)$ consisting of complexes of graded projective $A$-modules up to homotopy.

We also re-interpret chain complexes of graded $A$-modules as differential graded modules over the graded dg algebra~$(A, d)$, where~$A$ sits in homological degree zero, and the natural grading of~$A$ is orthogonal to the homological grading. A chain complex of graded $A$-modules is equivalent to the data of a differential graded $(A,d)$-module
\begin{gather*}M=\bigoplus_{j,k\in \Z}M^j_{k},\qquad d\big(M_{k}^j\big)\subset M_{k+1}^j.\end{gather*}
Extending the (inverse) equivalence of Proposition~\ref{prop-abelian-equivalence}, there is an auto-equivalence of abelian categories
\begin{gather*}
\mc{G}_\infty\colon \ (A,d)\dmod \lra \mc{M}_3,
\end{gather*}
where, on the object $\mc{G}_\infty(M)\in \mc{M}_3$ for a given $M\in (A,d)\dmod$, the generator $\dif_3$ acts by the differential $(-1)^k d\colon M_{k}\lra M_{k+1}$. It follows from Proposition~\ref{prop-abelian-equivalence} that $\mc{G}_\infty$ commutes with the translation by the various shift functors as follows:
\begin{gather}
 \mc{G}_\infty (M\langle 1 \rangle) = \mc{G}_\infty(M)\{1,1,0\}, \nonumber\\
 \mc{G}_\infty(M[1]) = \mc{G}_\infty (M)\{0,0,-1\}, \nonumber\\
 \mc{G}_\infty (\mc{T}(M)) = \mc{G}_\infty(M)\{1,0,0\}.\label{eqn-G-commute-with-grading-shifts}
\end{gather}
As a result, one can deduce that
\begin{gather}\label{eqn-G-on-Pr}
 \mc{G}_\infty(P_r\langle j \rangle[k])\cong Q\{r+j,j,-k\}.
\end{gather}

\begin{lem}\label{lem-contractible-dg-mod}
A morphism of chain complexes $f\colon M \lra N$ of graded $A$-modules is null-homotopic if and only if it factors through the canonical embedding of graded dg modules over~$A$
\[
\lambda_M \colon \ M\lra M\otimes \Bbbk[d]/\big(d^2\big)[1],\qquad m\mapsto m\otimes d.
\]
Consequently, $M$ is contractible if and only if $M$ is a dg summand of $M\otimes \Bbbk[d]/\big(d^2\big)[1]$.
\end{lem}

\begin{proof}Suppose $f=dh+hd\colon M\lra N$ is null-homotopic, with $h\colon M_{k}\lra N_{k-1}$ the null-homo\-topy map. We define
\[
\widehat{h}\colon \ M\otimes \Bbbk[d]/\big(d^2\big)[1]\lra N
\]
by, for any homogeneous $m\in M_{k}$,
\[
\widehat{h}(m\otimes 1):= (-1)^{k}h(m),\qquad \widehat{h}(m\otimes d):=dh(m)+hd(m).
\]
It is an easy exercise to check that $\widehat{h}$ is a map of dg $A$-modules. Then, we clearly have a~fac\-toriza\-tion $ f=\widehat{h}\circ \lambda_M $.

Conversely, if there is a factorization of dg $A$-modules
\[
\begin{gathered}
\xymatrix{
M\ar[dr]_-{\lambda_M} \ar[rr]^f && N \\
 & M\otimes \Bbbk[d]/\big(d^2\big)[1], \ar[ur]_{\widehat{h}} &
}
\end{gathered}
\]
define $h\colon M\lra N$ by $h(m):=(-1)^k\widehat{h}(m\otimes 1)$ for any $m\in M_{k}$. Another easy computation shows that $f=dh+hd$ is indeed null-homotopic.
\end{proof}

For the next result, for any graded dg module $M$ over $A$, denote by $M^\mu$ the corresponding dg module with the same underlying bigraded $A$-module as $M$, but the differential acting by zero instead. Under the equivalence $\mc{G}_\infty$, this corresponds to the $\mu$-pull-back (see equation~\eqref{eqn-augmentations}) of a trigraded $\Lambda_3$-module.

\begin{lem}\label{lem-iso-dg-contractible-mod}
Let $M$ be a graded dg module over $A$. There is an isomorphism of dg modules
\[
\phi_M\colon \ M\otimes \Bbbk[d]/\big(d^2\big) \lra M^\mu \otimes \Bbbk[d]/\big(d^2\big),
\]
defined by
\[
\phi_M(m\otimes 1):= m\otimes 1, \qquad \phi_M(m\otimes d):= (-1)^{k+1}d(m)\otimes 1+m\otimes d.
\]
for any $m\in M_{k}$.
\end{lem}
\begin{proof}It is an easy computation to verify that $\phi_M$ commutes with the respective differentials on both sides. The inverse of $\phi_M$ is given by
\[
\psi_M \colon \ M^\mu\otimes \Bbbk[d]/\big(d^2\big)\lra M\otimes \Bbbk[d]/\big(d^2\big),
\]
where, for $m\in M_{k}$,
\[
\psi_M(m\otimes 1):=m\otimes 1,\qquad \psi_{M}(m\otimes d):=(-1)^kd(m)\otimes 1+m\otimes d.
\]
Clearly, both $\phi_M$ and $\psi_M$ are homogeneous $A$-module maps. The lemma follows.
\end{proof}

\begin{cor}\label{cor-functor-G}
The functor $\mc{G}_\infty\colon (A,d)\dmod\lra \mc{M}_3$ descends to an exact functor
\[
\mc{G}\colon \ \mc{C}(A\pmod)\lra \mc{SM}_3.
\]
\end{cor}
\begin{proof}By Lemma \ref{lem-contractible-dg-mod}, $\mc{C}(A\pmod)$ is the categorical quotient of the category of chain complexes of graded projective $A$-modules by the ideal of morphisms factoring through objects of the form $M\otimes\Bbbk[d]/\big(d^2\big)$, where~$M$ ranges over all chain-complexes of graded projective $A$-modules. For~$\mc{G}$ to be well-defined, it suffices to check that, under the functor $\mc{G}_\infty$, such objects are sent to the class of projective-injective objects in~$\mc{M}_3$.

By Lemma \ref{lem-iso-dg-contractible-mod}, there is an isomorphism of graded dg modules
\[
M\otimes \Bbbk[d]/\big(d^2\big)\cong M^\mu \otimes \Bbbk[d]/\big(d^2\big)\cong \bigoplus_k M^\mu_k \otimes \Bbbk[d]/\big(d^2\big).
\]
As each $M^\mu_k$ is a projective $A$-module, the result follows since, by Proposition \ref{prop-abelian-equivalence},
\[
\mc{G}\big(P_{r}\langle j\rangle \otimes \Bbbk[d]/\big(d^2\big)[-k]\big)\cong Q\{r,r+j,k\}\otimes \Lambda_1^\prime \cong \Lambda_3\{r,r+j,k\}
\]
holds for any $r,j,k\in \Z$.

It remains to show that $\mc{G}$ is exact, i.e., it commutes with homological shifts and takes distinguished triangles to distinguished triangles.

Given a complex $M$ of graded projective modules over $A$, there is a short exact sequence
\[
0\lra M \stackrel{\lambda_M}{\lra} M\otimes \Bbbk[d]/\big(d^2\big)[1] \lra M[1]\lra 0.
\]
Applying $\mc{G}_\infty$ to the short exact sequence, we obtain a short exact sequence of $\mc{M}_3$:
\[
0\lra \mc{G}_\infty(M) \xrightarrow{ \mc{G}_\infty(\lambda_M)} \mc{G}_\infty \big(M\otimes \Bbbk[d]/\big(d^2\big)[1]\big) \lra \mc{G}_\infty( M[1])\lra 0,
\]
which, in turn, leads to a distinguished triangle in $\mc{SM}_3$:
\[
 \mc{G}(M) \xrightarrow{ \mc{G}(\lambda_M)} \mc{G} \big(M\otimes \Bbbk[d]/\big(d^2\big)[1]\big) \lra \mc{G}( M[1]) \stackrel{[1]}{\lra} \mc{G}(M)[1]_{\mc{SM}}.
\]
By the earlier discussion in this proof, the term $ \mc{G}_\infty \big(M\otimes \Bbbk[d]/\big(d^2\big)[1]\big)$ vanishes in $\mc{SM}_3$, and thus there is an isomorphism
\begin{gather*}
\mc{G}( M[1]) \cong \mc{G}(M)[1]_{\mc{SM}},
\end{gather*}
which is clearly functorial in~$M$.

Lastly, notice that distinguished triangles in~$\mc{C}(A\pmod)$, up to isomorphism, arise from short exact sequences of chain-complexes of graded projective $A$-modules
\[
0\lra M_1\lra M \lra M_2\lra 0.
\]
Applying $\mc{G}_\infty$ to this sequence, we obtain a short exact sequence of trigraded $\Lambda_3$-modules. This sequence results in a distinguished triangle in~$\mc{SM}_3$, being the image of the original triangle in~$\mc{C}(A\pmod)$. The exactness of $\mc{G}$ now follows.
\end{proof}

Denote by $\mc{C}^b(A\pmod)$, $\mc{C}^+(A\pmod)$ and $\mc{C}^-(A\pmod)$ the full triangulated subcategories of $\mc{C}(A\pmod)$ consisting of, respectively, bounded, bounded-from-below and bounded-from-above complexes of graded projective modules over $A$. The localization functor from~$\mc{C}(A\pmod)$ into~$\mc{D}(A)$ restricts to equivalences of categories on these full-subcategories onto their respective images in the~(dg) derived category $\mc{D}(A)$.

\begin{thm}\label{thm-faithful-embedding}
The functor $\mc{G}\colon \mc{C}(A\pmod)\lra \mc{SM}_3$ is fully-faithful.
\end{thm}
\begin{proof}The proof is divided into three steps.

As the first step, we claim that $\mc{G}$, when restricted to the full-subcategories $\mc{C}^\pm(A\pmod)$, is fully-faithful. To do this, we identify these categories with their images in $\mc{D}(A)$ under localization, and use the fact that the (dg) derived category of $(A, d)\dmod$ is compactly generated by the collection of objects $\{P_r\langle j\rangle[k]|r,j,k\in \Z\}$. Then, in order to prove the claim, we just need to compare the morphism spaces between the generating objects $P_{r}\langle j \rangle[k]$, $r,j,k\in \Z$, and their images $\mc{G}(P_{r}\langle j \rangle[k])=Q\{r+j,j,-k\}$ in $\mc{SM}_3$ \cite[Lemma~4.2]{K}.

On the one hand, we have
\begin{gather*}
 \Hom_{\mc{C}(A)} (P_{r_1}\langle j_1 \rangle[k_1], P_{r_2}\langle j_2 \rangle[k_2] )=
 \begin{cases}
 \Bbbk (r_1), & r_1=r_2,~j_1=j_2,~k_1=k_2,\\
 \Bbbk (r_1\,|\,r_1+1), & r_1=r_2+1,~j_1=j_2,~ k_1=k_2,\\
 \Bbbk (r_1\,|\,r_1-1), & r_1=r_2-1,~j_1=j_2-1,~k_1=k_2,\\
 \Bbbk (r_1\,|\,r_1+1|r_1), & r_1=r_2,~j_1=j_2+1,~ k_1=k_2,\\
 0, & \textrm{otherwise}.
 \end{cases}
\end{gather*}
On the other hand, we can compute the morphism spaces of $\mc{G}(P_r\langle j \rangle [k])$ using Lemma~\ref{lem-morphism-in-stable-cat}
\begin{gather*}
 \Hom_{\mc{SM}_3}\left(\mc{G}(P_{r_1}\langle j_1 \rangle[k_1]),\mc{G}( P_{r_2}\langle j_2 \rangle[k_2])\right) \\
 \qquad{} \cong \Hom_{\Lambda_3}(Q,Q\{r_2+j_2-r_1-j_1, j_2-j_1,k_1-k_2\}) \\
 \qquad{} \cong \Hom_{\Lambda_2}(Q,Q\{r_2+j_2-r_1-j_1, j_2-j_1,k_1-k_2\}).
\end{gather*}
The second isomorphism follows since $\dif_3$ acts trivially on $Q$ and its grading shifts. Note that the last space is non-zero only if $k_1=k_2$, since the $\Lambda_2$ action preserves the $k$-grading. When $k_1=k_2$, we can compute
\begin{gather*}
 \Hom_{\mc{SM}_3} (\mc{G}(P_{r_1}\langle j_1 \rangle[k_1])_{\mc{SM}},\mc{G}( P_{r_2}\langle j_2 \rangle[k_2]_{\mc{SM}}) ) \cong
 \begin{cases}
 \Bbbk, & r_1=r_2,~j_1=j_2,\\
 \Bbbk, & r_1=r_2+1,~j_1=j_2,\\
 \Bbbk, & r_1=r_2-1,~j_1=j_2-1,\\
 \Bbbk, & r_1=r_2,~j_1=j_2+1.
 \end{cases}
\end{gather*}
Comparing these computations, the claim follows.

In the second step, we show that
\[
\Hom_{\mc{C}(A)}\left(M,N\right)\stackrel{\mc{G}}{\lra}\Hom_{\mc{SM}_3}(\mc{G}(M),\mc{G}(N))
\]
is a bijection if one of $M$ or $N$ lies in $\mc{C}^\pm (A\pmod)$. Assume, for instance, that $M\in \mc{C}^+(A\pmod)$ and $N\in \mc{C}(A\pmod)$. Without loss of generality, we may assume that $M_k=0$ for all $k<0$. Then, $N$ fits into a distinguished triangle
\[
N_{\geq -1} \lra N \lra N_{\leq -2} \stackrel{[1]}{\lra} N_{\geq -1}[1],
\]
where $N_{\geq -1}$ is the subcomplex of $N$ of the form
\[
N_{\geq -1}=\Big(
\cdots \lra 0\lra N_{-1} \stackrel{d_{-1}}{\lra} N_0 \stackrel{d_0}{\lra} N_{1} \stackrel{d_1}{\lra} \cdots
\Big),
\]
and $N_{\leq -2}$ is the quotient complex
\[
N_{\leq -2}=\Big(
\cdots \stackrel{d_{-5}}{\lra} N_{-4} \stackrel{d_{-4}}{\lra} N_{-3} \stackrel{d_{-3}}{\lra} N_{-2} {\lra} 0 \lra \cdots \Big).
\]
It is readily seen that, in the homotopy category, we have
\[
\Hom_{\mc{C}(A)}(M,N_{<-1})=0, \qquad \Hom_{\mc{C}(A)}(M,N_{<-1}[-1])=0.
\]
Likewise, as the $k$-degrees of the objects $\mc{G}(N_{<-1})$ and $\mc{G}(N_{<{-1}}[-1])$ are bounded above by $-1$, we have that, by Lemma \ref{lem-morphism-in-stable-cat},
\begin{gather*}
\Hom_{\mc{SM}_3}(\mc{G}(M),\mc{G}(N_{\leq -2}))=0, \qquad \Hom_{\mc{SM}_3}(\mc{G}(M),\mc{G}(N_{\leq -2})[-1]_{\mc{SM}})=0.
\end{gather*}
Therefore, by the previous step, we have that
\begin{gather*}
\Hom_{\mc{C}(A)}(M,N) =\Hom_{\mc{C}(A)}(M,N_{\geq -1})\\
\qquad{} \xrightarrow{\cong} \ \Hom_{\mc{SM}_3}(\mc{G}(M),\mc{G}(N)) =\Hom_{\mc{SM}_3}(\mc{G}(M),\mc{G}(N_{\geq -1})).
\end{gather*}
The other cases are similar, and we leave them as exercises to the reader.

Finally, assume both $M$ and $N$ are any objects of $\mc{C}(A\pmod)$. We can truncate $N$ as
\[
N_{\geq 0} \lra N \lra N_{\leq -1} \stackrel{[1]}{\lra} N_{\geq 0}[1],
\]
Then, we have the commutative diagram
\[
\xymatrix@C=0.9em{
\cdots \ar[r] & \Hom_{\mc{C}(A)}(M, N_{\geq 0}) \ar[r] \ar[d]^{\mc{G}} & \Hom_{\mc{C}(A)}(M, N) \ar[r] \ar[d]^{\mc{G}} & \Hom_{\mc{C}(A)}(M, N_{\leq -1}) \ar[r] \ar[d]^{\mc{G}} & \cdots\\
\cdots \ar[r] & \Hom_{\mc{SM}_3}(\mc{G}(M),\mc{G}(N_{\geq 0})) \ar[r] & \Hom_{\mc{SM}_3}(\mc{G}(M),\mc{G}(N)) \ar[r] & \Hom_{\mc{SM}_3}(\mc{G}(M),\mc{G}(N_{\leq -1})) \ar[r] & \cdots
}
\]
The middle vertical arrow is then an isomorphism by the classical five lemma and the previous case. This finishes the proof of the theorem.
\end{proof}

In \cite{KhSe}, a braid group action on the homotopy category $\mc{C}(A)$ is introduced. The braid group generator $\mc{R}_r$ acts, on a chain complex $M$ of graded $A$-modules, by
\begin{gather*}
 \mc{R}_r(M):= \mathrm{C}\Big(P_r\otimes (r)M \stackrel{f}{\lra} M \Big),
\end{gather*}
where $f$ is the left $A$-module map determined by $(r)\otimes (r)m \mapsto (r)m$.

\begin{lem}\label{lem-G-intertwines-braid-group-action}
The functor $\mc{G}$ commutes with the braid group actions on $\mc{C}(A\pmod)$ and $\mc{SM}_3$.
\end{lem}
\begin{proof}It suffices to show that each $\mc{R}_r$ commutes with $\mc{G}$. This follows from equation~\eqref{eqn-G-on-Pr}, Proposition \ref{prop-abelian-equivalence} and the fact that $\mc{G}$ sends the cone construction in $\mc{C}(A\pmod)$ to that of the $\dif_3$-cone in $\mc{SM}_3$.
\end{proof}

\begin{cor}The action of $\mathrm{Br}_\infty$ on the category $\mc{SM}_3$ is faithful.
\end{cor}
\begin{proof}In \cite[Corollary 1.2]{KhSe}, it is shown that the braid group $\mathrm{Br}_{m+1}$ on $m+1$ strands acts faithfully on $ \mc{C}^b(A_m\pmod) $. The Corollary then follows from Theorem~\ref{thm-faithful-embedding}, Lemma \ref{lem-G-intertwines-braid-group-action} and taking the limit $m\to \infty$.
\end{proof}

\subsection*{Acknowledgements}
M.K.\ was partially supported by grants DMS-1406065, DMS-1664240, and DMS-1807425 from the NSF, while Y.Q.\ was partially supported by the NSF grant DMS-1947532 when working on this paper.

The first author learned algebraic topology for the first time from the Russian classic by Dmitry Fuchs and Anatolii Fomenko \cite{FF} (in its first edition named \emph{Homotopic Topology}), while the second author enjoyed teaching graduate courses out of this reprinted classic at his previous work institution. It is our great pleasure to dedicate this short note to Dmitry Fuchs on the occasion of his eightieth anniversary.

\pdfbookmark[1]{References}{ref}
\LastPageEnding

\end{document}